\documentclass{article}
\usepackage{times,authblk,graphicx}
\usepackage{upgreek}   
\usepackage{amsmath, amsfonts, amssymb, amsthm, bm, stmaryrd}
\usepackage[font=small,labelfont=bf]{caption}
\renewcommand{\theta}{\vartheta}
\renewcommand{\phi}{\varphi}
\renewcommand{\rho}{\varrho}
\renewcommand{\tau}{\uptau}

\renewcommand{\leq}{\leqslant}
\renewcommand{\geq}{\geqslant}

\newtheorem{Def}{Definition}[section]
\newenvironment{definition}{\begin{Def} \rm}{\end{Def}}
\newtheorem{lemma}[Def]{Lemma}
\newtheorem{proposition}[Def]{Proposition}

\newtheorem{theorem}[Def]{Theorem}

\newcommand{\AufzAnfang}{\begin{enumerate}}
\newcommand{\AufzEnde}{\end{enumerate}}
\newcommand{\Nummer}[1]{\item[{\rm (#1)}]\parindent0pt\parskip1ex}
\newcommand{\PunkteAnfang}{\begin{itemize}}
\newcommand{\PunkteEnde}{\end{itemize}}

\newcommand{\AxiomeAnfang}{\begin{description}}
\newcommand{\AxiomeEnde}{\end{description}}
\newcommand{\Axiom}[1]{\item[{\rm (#1)}]\parindent0pt\parskip1ex}

\newcommand{\komma}{,\hspace{0.6em}}

\newcommand{\Reals}{{\mathbb R}}

\newcommand{\ohne}{\backslash}

\newcommand{\equ}{\sim}
\newcommand{\equcl}[1]{\langle #1 \rangle}
\newcommand{\quot}{\sim}

\newcommand{\und}{\odot}
\newcommand{\impl}{\rightarrow}

\newcommand{\id}{\mbox{\sl id}}

\newcommand{\constant}[2]{\text{\sl c}^{#1,#2}}

\newcommand{\vonunten}{\nearrow}

\newcommand{\mitNull}[1]{{#1}^0}
\newcommand{\truncplus}{\oplus}
\newcommand{\fp}[1]{\text{\rm fp}(#1)}

\begin{document}

\title{Real coextensions as a tool for \\
constructing triangular norms
\thanks{Preprint of an article published by Elsevier in the {\sl Information Sciences} {\bf 348} (2016), 357-376. It is available online at: {\tt https://www.sciencedirect.com/science/article/pii/ S002002551630055X}.}}

\author{Thomas Vetterlein}

\affil{\footnotesize Department of Knowledge-Based Mathematical Systems \\
Johannes Kepler University Linz \\
Altenberger Stra\ss{}e 69, 4040 Linz, Austria \\
{\tt Thomas.Vetterlein@jku.at}}

\maketitle

\begin{abstract}\parindent0pt\parskip1ex

\mbox{}\vspace{-2ex}

We present in this paper a universal method of constructing left-continuous triangular norms (l.-c.\ t-norms). The starting point is an arbitrary, possibly finite, totally ordered monoid fulfilling the conditions that are characteristic for l.-c.\ t-norms: commutativity, negativity, and quanticity. We show that, under suitable conditions, we can extend this structure by substituting each element for a real interval. The process can be iterated and if the final structure obtained in this way is order-isomorphic to a closed real interval, its monoidal operation can, up to isomorphism, be identified with a l.-c.\ t-norm.

We specify the constituents needed for the construction in an explicit way. We furthermore illustrate the method on the basis of a number of examples.

\end{abstract}

\vspace{2ex}

\section{Introduction}

A fundamental issue in fuzzy set theory has been the question how the basic set-theoretical operations of intersection, union, and complement should be generalised to the case that membership comes in degrees; see, e.g., \cite[Chapter 10]{KlYu}. By general agreement, such operations should be defined pointwise. Indeed, this basic assumption is in accordance with the disjunctive interpretation of fuzzy sets and arguments on axiomatic grounds have been given, e.g., in \cite{Kle}. Otherwise, however, no standards exist and there are not even convincing arguments for restricting the possibilities to a manageable number in order to facilitate the choice.

To define the intersection of fuzzy sets, we are hence in need of a suitable binary operation on the real unit interval $[0,1]$. It is clear that only triangular norms (t-norms) come into question. The minimal requirements of a conjunction are then fulfilled: associativity, commutativity, neutrality w.r.t.\ $1$, and monotonicity in each argument. However, these properties are not very specific and t-norms exist in abundance. In fuzzy set theory, t-norms have consequently become a research field in its own right. A basic reference is the monograph \cite{KMP} and overviews are provided, e.g., in \cite{Mes1,Fod}. Among the numerous specialised studies on t-norms from recent times, we may mention, e.g., \cite{MaMe,FoRu}. The present paper is meant as a further contribution towards a better understanding of these operations.

Also in mathematical fuzzy logic, t-norms are employed for the interpretation of the conjunction. In this context, the implication is commonly assumed to be the adjoint of the conjunction; see, e.g., \cite{Haj1}. Given a t-norm, however, a residual implication does not necessarily exist; to this end, the t-norm must be in each argument left-continuous. Here, we generally assume this additional property to hold.

To explore t-norms, different perspectives can be chosen. Often t-norms have been studied as two-place real functions and geometric aspects have played a major role. For the sake of a classification, it makes sense not to distinguish between isomorphic operations. We recall that t-norms $\und_1$ and $\und_2$ are isomorphic if there is an order automorphism $\phi$ of the real unit interval such that $a \und_2 b = \phi^{-1}(\phi(a) \und_1 \phi(b))$ for any $a, b \in [0,1]$. In this case, it is reasonable to adopt an algebraic perspective. We will do so as well. Our aim is to classify, up to isomorphism, the structures $([0,1]; \leq, \und, 1)$, where $\leq$ is the natural order of the reals and $\und$ is a l.-c.\ t-norm.

The structures of the form $([0,1]; \leq, \und, 1)$ are totally ordered monoids, or tomonoids for short \cite{Gab,EKMMW}. They are, however, quite special among this type of algebras. The tomonoids in which we are interested are commutative, because so is each t-norm. Furthermore, they are negative, because the monoidal identity is the top element. Finally, the left-continuity of t-norms corresponds to a property that we call ``quantic''. In a word, the quantic, negative, and commutative tomonoids (q.n.c.\ tomonoids) whose base set is the real unit interval, are in a one-to-one correspondence with l.-c.\ t-norms.

Investigating tomonoids, we can profit from the enormous progress that the research on algebraic structures around left-continuous t-norms has made in recent times. A summary of results on residuated structures can be found, e.g., in \cite[Chapter 3]{GJKO}. In \cite{NEG2}, t-norms are especially taken into account. A number of further, more specialised overviews is contained in \cite{CHN}. We may in fact say that the situation as regards t-norms is today considerably more transparent than it used to be a few years ago.

To reveal the structure of a t-norm $\und$, the first natural step is to determine the quotients of the tomonoid based on $\und$. In fact, it has turned out that in this way the vast majority of t-norms known in the literature can be described in a uniform and transparent way. A systematic review of t-norms from this perspective was undertaken in \cite{Vet2} and in \cite{Vet3} the approach was applied in order to systematise a number of well-known t-norm construction methods.

Let us provide, on an intuitive basis, a summary of what follows. The quotients of tomonoids in which we are interested are constructed as follows. Let $\und$ be a l.-c.\ t-norm and assume that the set $F \subseteq [0,1]$ is (i) of the form $(d,1]$ or $[d,1]$ for a $d \in [0,1]$, and (ii) closed under multiplication, that is, $a \und b \in F$ for any $a, b \in F$. Then we call $F$ a filter. Define the equivalence relation $\equ_F$ by requiring $a \equ_F b$ if $a \und f \leq b$ and $b \und f \leq a$ for some $f \in F$. Then $\equ_F$ is a congruence of the tomonoid $([0,1]; \leq, \und, 1)$. This means that $[0,1]$ is partitioned into subintervals and $\und$ induces a binary operation making the set of these subintervals into a tomonoid again.

Consider, e.g., the t-norm $\und_H$ shown in Fig.\ \ref{fig:HajekschVier} (left), which is a modification of a t-norm defined by H\' ajek \cite{Haj2}. We depict $\und_H$ by indicating the mappings $[0,1] \to [0,1] \komma x \mapsto a \und_H x$ for several $a \in [0,1]$. We observe that $(\frac 3 4, 1]$ is closed under the operation $\und$ and hence a filter. Forming the quotient by $(\frac 3 4, 1]$ leads to the partition $\{0\}$, $(0, \frac 1 4]$, $(\frac 1 4$, $\frac 1 2]$, $(\frac 1 2, \frac 3 4]$, $(\frac 3 4, 1]$ of $[0,1]$. The operation $\und$ endows these five elements with the structure of the five-element \L ukasiewicz chain $L_5$; see Fig.\ \ref{fig:HajekschVier} (right).

\begin{figure}[ht!]
\begin{center}
\includegraphics[width=0.48\textwidth]{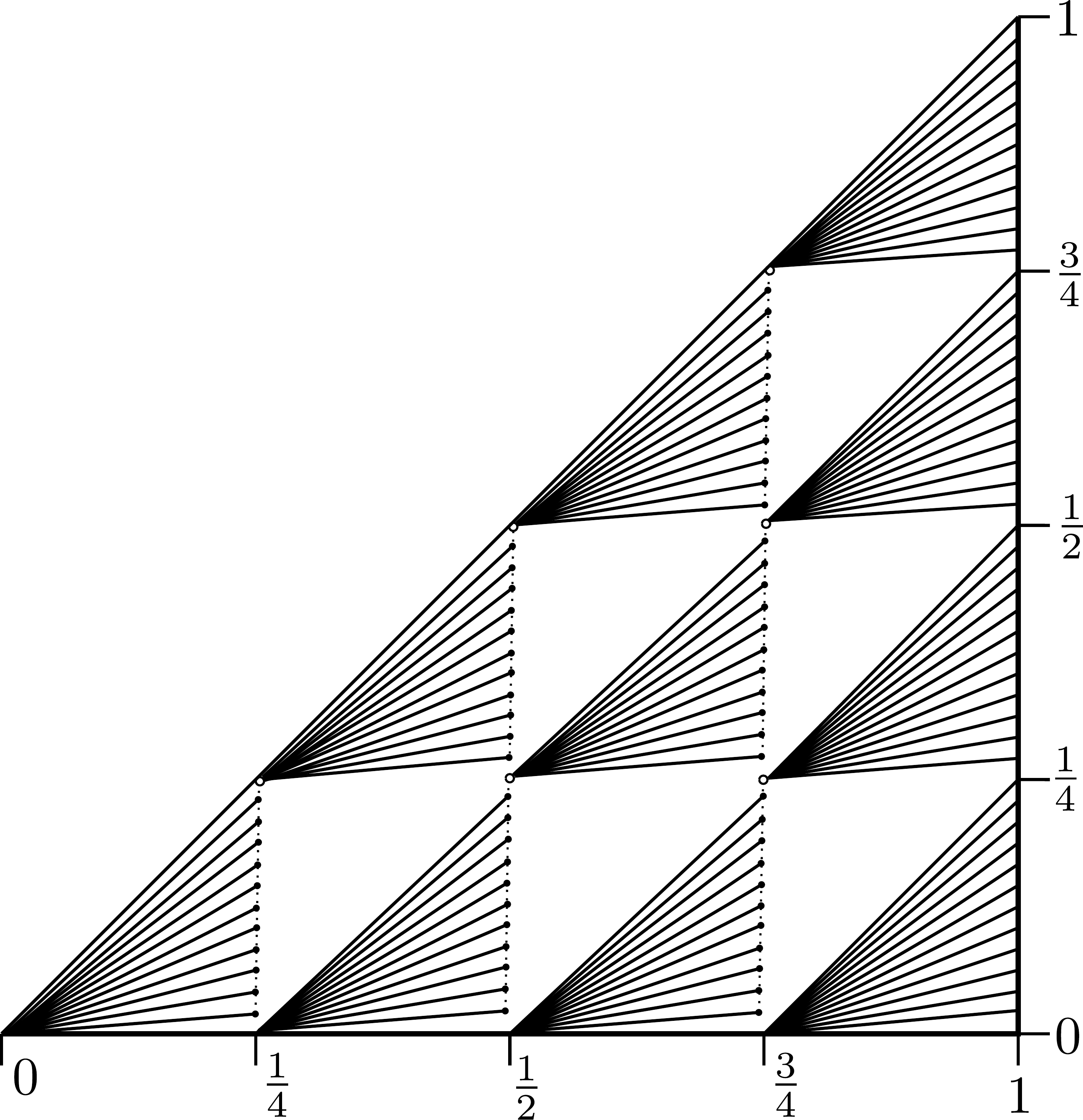}
\includegraphics[width=0.35\textwidth]{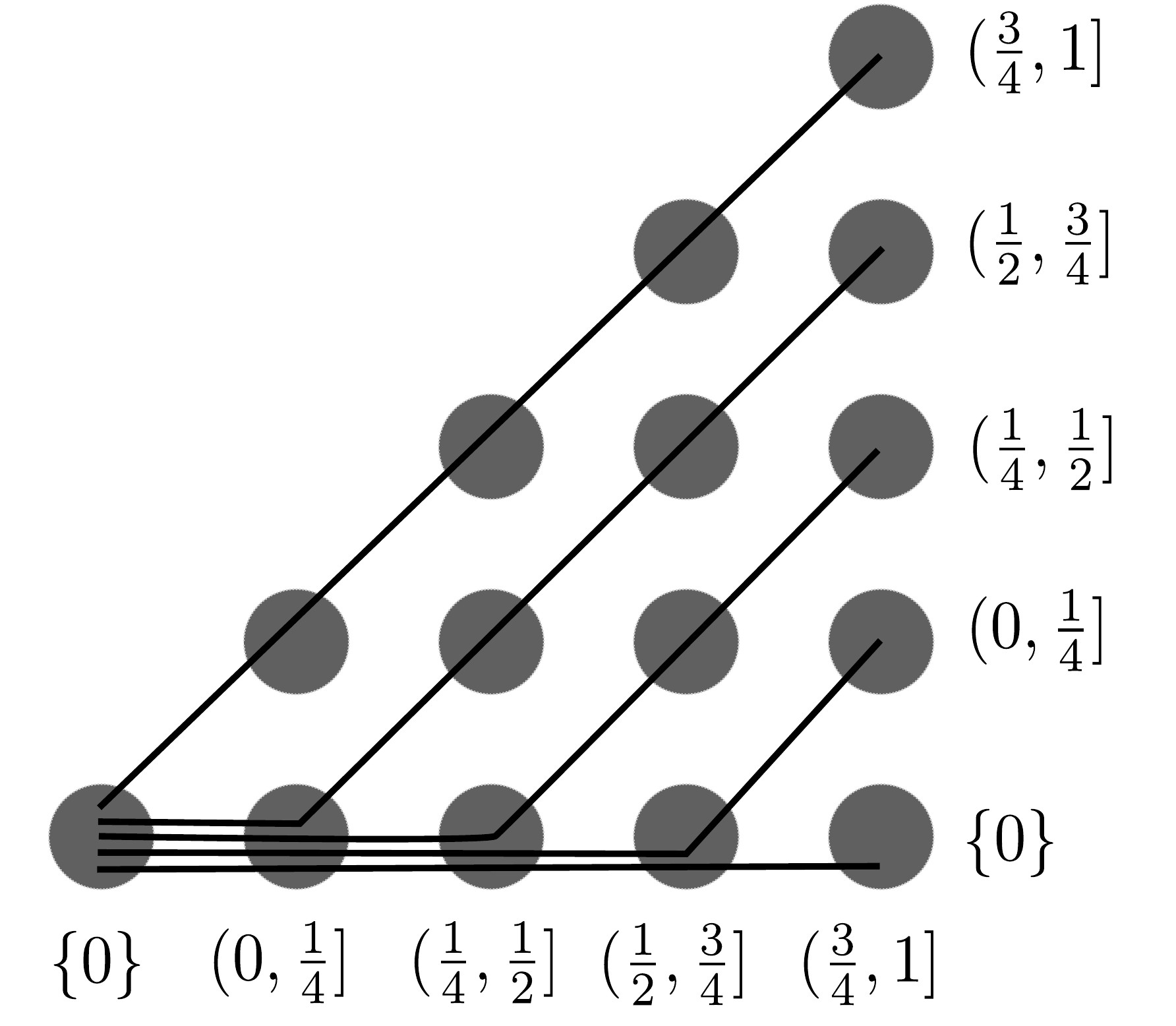}
\caption{Left: The t-norm $\und_H$. We show a selection of vertical cuts, that is, the multiplication with certain fixed elements. Right: The quotient w.r.t.\ $(\frac 3 4, 1]$ is the five-element \L ukasiewicz chain.}
\label{fig:HajekschVier}
\end{center}
\end{figure}

We hence see that a seemingly complex t-norm possesses a quotient that is as simple as the five-element \L ukasiewicz chain.

The question that we raise in the present work concerns the converse procedure. Let the tomonoid $L_5$ be given and assume that we want to expand each non-zero element to a left-open right-closed interval. Can we determine the monoidal operations on this enlarged universe such that the quotient is $L_5$? Are there any other operations apart from $\und_H$ with this property? Under suitable conditions, we will present in this paper a way of determining all possible monoidal operations on the enlarged universe. The t-norm $\und_H$ will turn out not to be the only solution and we easily determine the remaining ones as well.

Algebraically, our problem reads as follows. Let $(\mathcal P; \leq, \und, 1)$ be a q.n.c.\ tomonoid. Extend the chain $\mathcal P$ by substituting each element for a left-open or left-closed, right-open or right-closed real interval. Let $\mathcal L$ be the enlarged chain and let $F$ be the subset of $\mathcal L$ that has been chosen to replace the top element of $\mathcal P$. Our aim is to make $\mathcal L$ into a q.n.c.\ tomonoid of which $F$ is a filter and whose quotient by $F$ is the original tomonoid $\mathcal P$. To this end, we consider two situations. First, we assume that the filter $F$ is Archimedean. Adding to $F$ a least element if necessary, $F$ is then isomorphic to the tomonoid based on the \L ukasiewicz or product t-norm. In the second case we assume that $F$ is a semilattice. This in turn means that the monoidal operation is the minimum.

We will see that in both cases, there is not much room for variation. Precisely speaking, however, the monoidal operation on the extended universe is in general not uniquely determined. If not, a set of parameters that has at most the cardinality of $\mathcal P^2$ is sufficient to remove the ambiguities.

The paper is structured as follows. We compile in Section \ref{sec:definitions} the necessary basic facts about totally ordered monoids and their quotients. Section \ref{sec:composition-tomonoids} shortly introduces to composition tomonoids, a tool of geometric nature on which our specification of the extensions will rely. The main results are contained in the subsequent two parts. Namely, Section \ref{sec:Archimedean-coextensions} deals with what we call Archimedean coextensions and Section \ref{sec:semilattice-coextensions} deals with semilattice coextensions. To see how the results apply in practice, several examples of the construction of t-norms are provided in Section \ref{sec:examples}. An outlook to possible further research in the present field can be found in the concluding Section \ref{sec:conclusion}.

\section{T-norms, tomonoids and coextensions}
\label{sec:definitions}

We study in this paper the following type of functions.

\begin{definition}
A binary operation $\und$ on the real unit interval $[0,1]$ is called a {\it triangular norm}, or {\it t-norm} for short, if (i) $\und$ is associative, (ii) $\und$ is commutative, (iii) $a \und 1 = a$ for any $a \in [0,1]$, and (iv) $a \leq b$ implies $a \und c \leq b \und c$ for any $a, b, c \in [0,1]$.

Moreover, a t-norm $\und$ is called {\it left-continuous}, abbreviated {\it l.-c.}, if $\lim_{x \vonunten a} x \und b = a \und b$ for any $a \in (0,1]$ and $b \in [0,1]$.
\end{definition}

T-norms are commonly viewed as a many-valued analogue of the classical conjunction. Moreover, left-continuity is equivalent to the existence of a residual implication, hence this property is important in fuzzy logic \cite{Haj1}.

To classify t-norms, it is useful to choose an appropriate algebraic framework. There are several possibilities of doing so. T-norms can be identified, for instance, with certain MTL-algebras \cite{NEG2}. They can also be viewed as strictly two-sided commutative quantales \cite{Ros}.

The subsequent Definition \ref{def:tomonoid} represents our choice \cite{EKMMW,Gab}. Here, a poset is meant to be {\it almost complete} if the suprema of all non-empty subsets exist.

\begin{definition} \label{def:tomonoid}
A structure $({\mathcal L}; \leq, \und, 1)$ is called a {\it totally ordered monoid}, or {\it tomonoid} for short, if (i) $({\mathcal L}; \und, 1)$ is a monoid and (ii) $\leq$ is a total order on $\mathcal L$ that is compatible with $\und$, that is, for all $a, b, c, d \in {\mathcal L}$, $\;a \leq b$ and $c \leq d$ imply $a \und c \leq b \und d$.

Moreover, a tomonoid $\mathcal L$ is called {\it commutative} if so is $\und$; $\mathcal L$ is called {\it negative} if $a \leq 1$ for all $a \in {\mathcal L}$; and $\mathcal L$ is called {\it quantic} if (i) $\mathcal L$ is almost complete and (ii) for any elements $a, b_\iota$, $\iota \in I$, of $\mathcal L$ we have
\[ {\textstyle a \und \bigvee_\iota b_\iota \;=\; \bigvee_\iota (a \und b_\iota) \quad\text{and}\quad
(\bigvee_\iota b_\iota) \und a \;=\; \bigvee_\iota (b_\iota \und a).} \]
\end{definition}

We observe that the indicated properties largely coincide with those of t-norms. Hence we easily verify that these structures correspond to the operations in which we are interested.

\begin{lemma} \label{lem:t-norms-tomonoids}
Let $[0,1]$ be the real unit interval endowed with the natural order. Then a binary operation $\und \colon [0,1]^2 \to [0,1]$ is a t-norm if and only if $([0,1]; \leq, \und, 1)$ is a negative, commutative tomonoid. In this case, the t-norm is left-continuous if and only if the tomonoid is quantic.
\end{lemma}

We will abbreviate ``quantic, negative, commutative'' by ``q.n.c.''. By Lemma \ref{lem:t-norms-tomonoids}, l.-c.\ t-norms are in a one-to-one correspondence with q.n.c.\ tomonoids whose base set is the real unit interval. We call a tomonoid of this form a {\it t-norm monoid}.

Note that a q.n.c.\ tomonoid does not in general possess a bottom element. But we can, if necessary, add a zero in the usual way. The result is a q.n.c.\ tomonoid again, which is complete and can thus be seen as a quantale \cite{Ros}.

\begin{definition}
Let $({\mathcal L}; \leq, \und, 1)$ be a q.n.c.\ tomonoid. Let $\mitNull{\mathcal L} = {\mathcal L}$ if $\mathcal L$ has a bottom element. Otherwise, let $\mitNull{\mathcal L}$ arise from $\mathcal L$ by adding a new element $0$; extend the total order to $\mitNull{\mathcal L}$ such that $0$ is the bottom element; and extend $\und$ to $\mitNull{\mathcal L}$ such that $0 \und a = a \und 0 = 0$ for any $a \in \mitNull{\mathcal L}$.
\end{definition}

In the sequel, we always tacitly assume a q.n.c.\ tomonoid $\mathcal L$ to be a subset of $\mitNull{\mathcal L}$. In particular, infima are always meant to be calculated in $\mitNull{\mathcal L}$. The symbol $0$ denotes the bottom element of $\mitNull{\mathcal L}$.

A {\it subtomonoid} of a q.n.c.\ tomonoid $\mathcal L$ is a submonoid $F$ of $\mathcal L$ together with the total order restricted from $\mathcal L$ to $F$. By an {\it interval} of a q.n.c.\ tomonoid $\mathcal L$, we mean a non-empty subset $J$ of $\mathcal L$ such that $a, b \in J$ and $a \leq c \leq b$ imply $c \in J$. As $\mitNull{\mathcal L}$ is complete, any interval $J$ of $\mathcal L$ possesses a lower boundary $u = \inf J \in \mitNull{\mathcal L}$ and an upper boundary $v = \sup J \in {\mathcal L}$. We will denote intervals as is common for $\Reals$; for instance, $(u,v]$ denotes an interval $J$ such that $u = \inf J$ but $u \notin J$, and $v = \max J$.

A {\it homomorphism} between tomonoids is defined as expected. Furthermore, a mapping $\chi \colon A \to B$ between totally ordered sets $A$ and $B$ is called {\it sup-preserving} if, whenever the supremum of elements $a_\iota \in A$, $\iota \in I$, exists in $A$, then $\bigvee_\iota \chi(a_\iota) = \chi(\bigvee_\iota a_\iota)$ in $B$. Note that a sup-preserving mapping is in particular order-preserving.

Finally, any q.n.c.\ tomonoid $\mathcal L$ is residuated; we can define $a \impl b = \max \; \{ c \in {\mathcal L} \colon a \und c \leq b \}$ for $a, b \in {\mathcal L}$. Let $r, t \in {\mathcal L}$ and $s = r \und t$. We will call the pair $r, t$ {\it $\und$-maximal} if the following condition is fulfilled: $r$ is the largest element $x$ such that $s = x \und t$ and $t$ is the largest element $y$ such that $s = r \und y$. We note that for any pair $r, t \in {\mathcal L}$, there is a $\und$-maximal pair $\bar r, \bar t$ such that $\bar r \geq r$, $\,\bar t \geq t$, and $\bar r \und \bar t = r \und t$; for instance, let $s = r \und t$ and take $\bar r = t \impl s$ and $\bar t = \bar r \impl s$.

We now turn to quotients of tomonoids.

\begin{definition} \label{def:tomonoidcongruence}
Let $({\mathcal L}; \leq, \und, 1)$ be a q.n.c.\ tomonoid. A {\it tomonoid congruence} on $\mathcal L$ is a congruence $\equ$ of $\mathcal L$ as a monoid such that the $\equ$-classes are intervals. We endow the quotient $\equcl{\mathcal L}_\equ$ with the total order given by
\[ \equcl{a}_\equ \leq \equcl{b}_\equ \text{ if } 
   a' \leq b' \text{ for some } a' \equ a \text{ and } b' \equ b \]
for $a, b \in {\mathcal L}$, with the induced operation $\und$, and with the constant $\equcl{1}_\equ$. The resulting structure $(\equcl{\mathcal L}_\equ; \leq, \und, \equcl{1}_\equ)$ is called a {\it tomonoid quotient} of $\mathcal L$.
\end{definition}

There is no easy way of describing the quotients of q.n.c.\ tomonoids in a systematic way. There is, however, one well-known way of constructing a quotient: by means of a filter. These congruences are precisely those that preserve the residual implication as well. See, e.g., \cite{BlTs} for the more general case of residuated lattices and \cite{NEG1} for the case of MTL-algebras.

\begin{definition} \label{def:filter}
Let $({\mathcal L}; \leq, \und, 1)$ be a q.n.c.\ tomonoid. Then a {\it filter} of $\mathcal L$ is a subtomon\-oid $(F; \leq, \und, 1)$ of $\mathcal L$ such that $f \in F$ and $g \geq f$ imply $g \in F$.

If $F$ is a filter, let, for $a, b \in {\mathcal L}$,
\begin{align*}
a \quot_F b & \quad\text{if $a = b$,} \\
& \quad\text{or $a < b$ and there is a $f \in F$ such that $b \und f \leq a$,} \\
& \quad\text{or $b < a$ and there is a $f \in F$ such that $a \und f \leq b$.}
\end{align*}
Then we call $\quot_F$ the {\it congruence induced by $F$}.
\end{definition}

By the {\it trivial} tomonoid, we mean the one-element tomonoid, consisting of $1$ alone. Each non-trivial tomonoid $\mathcal L$ possesses at least two filters: $\{1\}$, the {\it trivial} filter, and $\mathcal L$, the {\it improper} filter. Let $d$ be the infimum of a filter $F$ of some q.n.c.\ tomonoid $\mathcal L$; then $F$ consists of all elements strictly larger than $d$ and possibly also $d$. If $d$ belongs to $F$, we shall write $F = [d, 1] = d^\leq$ and otherwise $F = (d,1] = d^<$.

Each filter of a q.n.c.\ tomonoid is again a q.n.c.\ tomonoid. We note that in order to ensure this fact, we have chosen part (i) of the definition of quanticity; indeed, each non-empty subset of a filter clearly possesses a supremum, but a filter need not have a bottom element.

Furthermore, a filter induces a quotient and all three properties that we consider here are preserved \cite{BlTs,Vet2}.

\begin{lemma} \label{lem:filter-quotient}
Let $({\mathcal L}; \leq, \und, 1)$ be a q.n.c.\ tomonoid and let $(F; \leq, \und, 1)$ be a filter of $\mathcal L$. Then the congruence induced by $F$ is a tomonoid congruence, and $\equcl{\mathcal L}_{\equ_F}$ is q.n.c.\ again.
\end{lemma}

\begin{proof}
This is obvious except for the fact that the quotient is again quantic. For this latter fact, see \cite{Vet2}.
\end{proof}

\begin{definition} \label{def:coextension}
Let $({\mathcal L}; \leq, \und, 1)$ be a q.n.c.\ tomonoid. Let $(F; \leq, \und, 1)$ be a filter of $\mathcal L$ and let $\quot_F$ be the congruence induced by $F$. Then we refer to the congruences classes as {\it $F$-classes}. Let $\mathcal P$ be the quotient of $\mathcal L$ by $\quot_F$. Then we call $\mathcal P$ the {\it quotient of $\mathcal L$ by $F$}. Furthermore, we call $\mathcal L$ a {\it coextension of $\mathcal P$ by $F$} and we refer to $F$ as the {\it extending} tomonoid.
\end{definition}

The aim of this paper is to present ways of constructing coextensions of q.n.c.\ tomonoids. Let us add a comment on our terminology. There are two basic ways of enlarging an algebra $A$ to an algebra $B$. Either we might want to determine algebras $B$ such that $A$ is (isomorphic to) a quotient of $B$; or we might require that $A$ is a subalgebra of $B$. In the latter case, $B$ is customarily called an extension of $A$. For instance, ideal extensions of semigroups have been widely studied; see, e.g., \cite{ClPr}. Here, we deal with the former case, where the situation with regard to the terminology is not so clear. Indeed, $A$ is either called an extension as well, or a coextension. For the sake of clarity and in accordance with Grillet's monograph on semigroups \cite{Gri}, we have opted for the term ``coextension'', adapted in the natural way to the case that an additional order is present.

We can say the following about the congruences induced by filters.

\begin{lemma} \label{lem:filters-tnorm}
Let $\mathcal L$ be a q.n.c.\ tomonoid and let $d \in \mitNull{\mathcal L}$.
\AufzAnfang
\Nummer{i} $d^\leq = [d,1]$ is a filter if and only if $d$ is an idempotent element of $\mathcal L$.

In this case, each $d^\leq$-class is of the form $[u,v]$ for some $u, v \in {\mathcal L}$ such that $u \leq v$. The $d^\leq$-class containing $1$ is $[d,1]$.

\Nummer{ii} $d^< = (d,1]$ is a filter if and only if $d \neq 1$, $\,d = \bigwedge_{a > d} a$, and $d < a \und b$ for all $a, b > d$.

In this case, each $d^<$-class is of the form $(u,v)$, $(u,v]$, $[u,v)$, or $[u,v]$ for some $u, v \in \mitNull{\mathcal L}$ such that $u < v$, or $\{u\}$ for some $u \in {\mathcal L}$. The $d^<$-class containing $1$ is $(d,1]$.
\AufzEnde
\end{lemma}

\begin{proof}
The interval $[d,1]$ is a filter if and only if $[d,1]$ is closed under multiplication. By the negativity of $\mathcal L$, this is the case if and only if $d \und d = d$, that is, if $d$ is an idempotent element. The first part of (i) follows.

Furthermore, $d$ can be the lower bound of an interval only if $d < 1$ and $d = \bigwedge_{a > d} a$. In this case, $(d,1]$ is a filter if and only if $(d,1]$ is closed under multiplication, that is, if $a, b > d$ implies $a \und b > d$. Thus also the first part of (ii) is clear.

Consider next an arbitrary filter $F$ of $\mathcal L$ and let $R$ be an $F$-class. By definition of a tomonoid congruence, $R$ is an interval. Let $u = \inf R$ and $v = \sup R$. Then, depending on whether or not $u$ and $v$ belong to $R$, we see that $R$ has one of the forms indicated in the last paragraph of the lemma. The second part of (ii) is shown.

To complete the proof, assume now that $F$ possesses the smallest element $d$. Let $a \in R$. Then, for $b \leq a$, we have that $b \in R$ if and only if $a \und d \leq b$, hence $R$ possesses the smallest element $u = a \und d$. Furthermore, for any $b \geq a$, we have that $b \in R$ if and only if $b \und d \leq a$. Since $\mathcal L$ is quantic, $\{ b \in \mathcal L \colon b \geq a \text{ and } b \und d \leq a \}$ has a maximal element $v$. Hence $R = [u,v]$ and also the first part of (i) follows.
\end{proof}

We focus in this paper on coextensions of q.n.c.\ tomonoids by filters and we are especially interested in the description of how a t-norm monoid arises from one of its quotients. But if a coextension is a t-norm monoid, each congruence class is a subinterval of $[0,1]$. For this reason we are motivated to focus on the following particular type of coextensions.

In what follows, a singleton is meant to be a set consisting of exactly one element. Moreover, by a real interval we mean either a singleton or a subset of the reals of the form $(a,b)$, $(a,b]$, $[a,b)$, or $[a,b]$ for some $a, b \in \Reals$ such that $a < b$. In addition, we use the real sets $\Reals^- = \{ r \in \Reals \colon r \leq 0 \}$ and $\Reals^+ = \{ r \in \Reals \colon r \geq 0 \}$.

\begin{definition} \label{def:real-coextension}
Let $\mathcal P$ be the quotient of the q.n.c.\ tomonoid $({\mathcal L}; \leq, \und, 1)$ by the filter $F$ of $\mathcal L$. Assume that each $F$-class is order-isomorphic to a real interval. Then we call $\mathcal L$ a {\it real coextension} of $\mathcal P$.
\end{definition}

We will furthermore impose certain conditions on the extending filter. We consider two contrasting cases: we assume that the extending tomonoid is either Archimedean or a semilattice. 

For an element $a$ of a tomonoid and $n \geq 1$, we will write $a^n = a \und \ldots \und a$ ($n$ factors).

\begin{definition} \label{def:Archimedean}
Let $({\mathcal L}; \leq, \und, 1)$ be a q.n.c.\ tomonoid. Then $\mathcal L$ is called \emph{Archimedean} if, for each $a, b \in {\mathcal L}$ such that $a < b < 1$, we have $b^n \leq a$ for some $n \geq 1$. A coextension of a q.n.c.\ tomonoid by an Archimedean tomonoid is called \emph{Archimedean} as well.

Furthermore, $\mathcal L$ is called a \emph{semilattice} if the monoidal product is the minimum, that is, if for any $a, b \in {\mathcal L}$ we have $a \und b = a \wedge b$. A coextension of a q.n.c.\ tomonoid by a semilattice is called \emph{semilattice} as well.
\end{definition}

We note that the Archimedean property is defined in a way that the monoidal identity is disregarded; otherwise it would only apply to the trivial tomonoid. On the other hand, the bottom element, if it exists, is not assumed to play an extra role, in contrast to the definition given, e.g., in \cite{KMP}. We further note that the definition of a semilattice is usually given for partially ordered monoids and the monoidal product is required to be the infimum. This is what the notion in fact suggests. As we deal with a total order, the infimum is the minimum; but we keep the common terminology.

\section{Composition tomonoids}
\label{sec:composition-tomonoids}

Each t-norm $\und$ can obviously be identified with the collection of its vertical cuts. Knowing, for each $a \in [0,1]$, the mapping $\lambda_a \colon [0,1] \to [0,1] \komma x \mapsto a \und x$ is equivalent to knowing $\und$ itself. Under the correspondence $a \mapsto \lambda_a$, we may in fact identify the t-norm monoid $([0,1]; \leq, \und, 1)$ with $\Lambda = \{ \lambda_a \colon a \in [0,1] \}$. The total order of $[0,1]$ then corresponds to the pointwise order of $\Lambda$; the operation $\und$ on $[0,1]$ becomes the functional composition on $\Lambda$; and the constant $1$ corresponds to the identity mapping.

More generally, each totally ordered monoid may be viewed as an $S$-poset. An $S$-poset consists of order-preserving mappings of some poset to itself and is endowed with the functional composition as a binary operation \cite{BuMa}. As in our previous papers \cite{Vet2,Vet3}, we will make essential use of the representation of tomonoids by $S$-posets. We will not use the notion of an $S$-poset, however; the following definition will be applied.

\begin{definition} \label{def:composition-tomonoid}
Let $(R; \leq)$ be a chain, and let $\Phi$ be a set of order-preserving mappings from $R$ to $R$. We denote by $\leq$ the pointwise order on $\Phi$, by $\circ$ the functional composition, and by $\id_R$ the identity mapping on $R$. Assume that {\rm (i)} $\leq$ is a total order on $\Phi$, {\rm (ii)} $\Phi$ is closed under $\circ$, and {\rm (iii)} $\id_R \in \Phi$. Then we call $(\Phi; \leq, \circ, \id_R)$ a \emph{composition tomonoid} on $R$.
\end{definition}

In what follows, commutativity of mappings refers to their functional composition. That is, two mappings $\phi, \psi \colon A \to A$ {\it commute} if $\phi \circ \psi = \psi \circ \phi$. Furthermore, a mapping $\phi$ of a poset $A$ to itself is called {\it shrinking} if $\phi(a) \leq a$ for all $a \in A$.

\begin{proposition} \label{prop:composition-tomonoid}
Let $(\Phi; \leq, \circ, \id_R)$ be a composition tomonoid over a chain $R$. Then $\Phi$ is a tomonoid.

Furthermore, $\Phi$ is commutative if and only if we have:
\AxiomeAnfang
\Axiom{C1} Any two mappings $\lambda_1, \lambda_2 \in \Phi$ commute.
\AxiomeEnde
$\Phi$ is negative if and only if we have:
\AxiomeAnfang
\Axiom{C2} Any mapping $\lambda \in \Phi$ is shrinking.
\AxiomeEnde
Finally, $\Phi$ is quantic if the following conditions hold:
\AxiomeAnfang
\Axiom{C3} Each $\lambda \in \Phi$ is sup-preserving.

\Axiom{C4} The pointwise calculated supremum of any non-empty subset of $\Phi$ exists and is in $\Phi$.
\AxiomeEnde
\end{proposition}

\begin{proof}
The fact that $\Phi$ is a tomonoid is easily checked. Moreover, the indicated characterisation of commutativity and negativity of $\Phi$ is obvious.

Assume (C3) and (C4). Then any non-empty subset of $\Phi$ possesses by (C4) w.r.t.\ the pointwise order a supremum; that is, $\Phi$ is almost complete. Furthermore, let $\lambda_\iota, \mu \in \Phi$, $\iota \in I$. Then we have by (C4) for any $r \in R$
\[ ((\bigvee_\iota \lambda_\iota) \circ \mu)(r)
\,=\, (\bigvee_\iota \lambda_\iota)(\mu(r))
\,=\, \bigvee_\iota \lambda_\iota(\mu(r))
\,=\, \bigvee_\iota (\lambda_\iota \circ \mu)(r)
\,=\, (\bigvee_\iota (\lambda_\iota \circ \mu))(r). \]
Moreover, we have by (C3) and (C4) for any $r \in R$
\[ (\mu \circ \bigvee_\iota \lambda_\iota)(r)
\;=\; \mu(\bigvee_\iota \lambda_\iota(r))
\;=\; \bigvee_\iota \mu(\lambda_\iota(r))
\;=\; \bigvee_\iota(\mu \circ \lambda_\iota)(r)
\;=\; (\bigvee_\iota(\mu \circ \lambda_\iota))(r). \]
We conclude that $\Phi$ is quantic.
\end{proof}

For composition tomonoids, two different kinds of isomorphisms exist, which we have to distinguish carefully. Let $(\Phi; \leq, \circ, \id_R)$ be a composition tomonoid over a chain $R$. Then we say that $\Phi$ is {\it c-isomorphic} to a further composition tomonoid $(\Psi; \leq, \circ, \id_S)$ over a chain $S$ if there is an order isomorphism $\phi \colon R \to S$ such that $\Psi = \{ \phi \circ \lambda \circ \phi^{-1} \colon \lambda \in \Phi \}$. Note that a c-isomorphism is a bijection between the underlying sets and does not directly refer to a monoidal structure.

Moreover, we apply the usual notion of an isomorphism between tomonoids also to composition tomonoids. Namely, a tomonoid $({\mathcal L}; \leq, \und, 1)$ is called {\it isomorphic} to $(\Phi; \leq, \circ, \id_R)$ if there is a bijective mapping $\pi \colon \mathcal L \to \Phi$ such that, for any $a, b \in {\mathcal L}$, $\;a \leq b$ iff $\pi(a) \leq \pi(b)$, $\;\pi(a \und b) = \pi(a) \circ \pi(b)$, and $\pi(1) = \id_R$. Note that an isomorphism is a mapping between tomonoids and does not depend on whether or not the tomonoid is represented by mappings.

Clearly, a c-isomorphism between two composition tomonoids gives rise to an isomorphism. The converse, however, does not hold. Consider the composition tomonoids $\Phi = \{ \lambda_a \colon a \in \Reals^- \}$, where $\lambda_a \colon \Reals^- \to \Reals^- \komma x \mapsto x + a$, and $\Psi = \{ \bar\lambda_a \colon a \in \Reals^- \}$, where $\bar\lambda_a \colon \Reals \to \Reals \komma x \mapsto x + a$. As $\Reals^-$ and $\Reals$ are not order-isomorphic, $\Phi$ and $\Psi$ are not c-isomorphic. However, as tomonoids, $\Phi$ and $\Psi$ are isomorphic. Indeed, both are isomorphic to $(\Reals^-; \leq, +, 0)$.

We next recall that each tomonoid can be viewed as a composition tomonoid. In the unordered case, this is simply the regular representation of monoids \cite{ClPr}.

\begin{proposition} \label{prop:tomonoid-composition-tomonoid}
Let $({\mathcal L}; \leq, \und, 1)$ be a q.n.c.\ tomonoid. For each $a \in {\mathcal L}$, put
\begin{equation} \label{fml:translation}
\lambda_a \colon {\mathcal L} \to {\mathcal L} \komma x \mapsto a \und x,
\end{equation}
and let $\Lambda = \{ \lambda_a \colon a \in {\mathcal L} \}$. Then $(\Lambda; \leq, \circ, \id_{\mathcal L})$ is a composition tomonoid on $\mathcal L$ fulfilling {\rm (C1)}--{\rm (C4)}. Moreover,
\begin{equation} \label{fml:Cayley}
\pi \colon {\mathcal L} \to \Lambda \komma a \mapsto \lambda_a
\end{equation}
is an isomorphism of the tomonoids $({\mathcal L}; \leq, \und, 1)$ and $(\Lambda; \leq, \circ, \id_{\mathcal L})$.
\end{proposition}

\begin{proof}[Proof (sketched).]
Since $\leq$ is a compatible total order, the mappings $\lambda_a$ are order-preserving. Moreover, for any $a, b \in {\mathcal L}$ such that $a \leq b$, we have $\lambda_a(x) = a \und x \leq b \und x = \lambda_b(x)$ for all $x \in {\mathcal L}$. We conclude that $\lambda_a \leq \lambda_b$ if and only if $a \leq b$; in particular, the (pointwise) order on $\Lambda$ is a total order. Furthermore, for any $a, b \in {\mathcal L}$, we have $\lambda_a(\lambda_b(x)) = \lambda_{a \und b}(x)$ for all $x \in {\mathcal L}$, that is, $\Lambda$ is closed under $\circ$. Finally, $\lambda_1 = \id_{\mathcal L}$. We have shown that $(\Lambda; \leq, \circ, \id_{\mathcal L})$ is a composition tomonoid.

Since $\und$ is associative and commutative, $\lambda_a \circ \lambda_b = \lambda_b \circ \lambda_a$ for any $a, b \in {\mathcal L}$, that is, (C1) holds. Furthermore, by negativity, $\lambda_a(x) \leq x$ for any $a, x \in {\mathcal L}$, that is, (C2) holds. Since $\mathcal L$ is quantic, $\lambda_a$ is sup-preserving for all $a$, that is, (C3) holds. Finally, let $A$ be a non-empty subset of $\mathcal L$ and let $b = \sup A$. Then, by quanticity, $\bigvee_{a \in A} \lambda_a(x) = \bigvee_{a \in A} (a \und x) = (\bigvee_{a \in A} a) \und x = b \und x = \lambda_b(x)$ for all $x \in {\mathcal L}$. Hence $\lambda_b$ is the pointwise supremum of the $\lambda_a$, $a \in A$, and (C4) follows.

Finally, we easily check that $\pi$, given by (\ref{fml:Cayley}), is a homomorphism. Moreover, $\pi$ is surjective by construction and $\pi$ is injective because, for $a, b \in \mathcal L$, $\lambda_a = \lambda_b$ implies $a = \lambda_a(1) = \lambda_b(1) = b$. Hence $\pi$ is indeed an isomorphism.
\end{proof}

Given a tomonoid $({\mathcal L}; \leq, \und, 1)$, each mapping (\ref{fml:translation}) is called a {\it left inner translation}, or just a {\it translation} for short. Moreover, we will call the composition tomonoid $(\Lambda; \leq,$ $\circ, \id_{\mathcal L})$ consisting of the translations of $\mathcal L$ the \emph{Cayley tomonoid} of $\mathcal L$.

We next consider a q.n.c.\ tomonoid $\mathcal L$ together with its quotient by a filter $F$. The Cayley tomonoid $\Lambda$ of $\mathcal L$ describes the tomonoid as whole. We will see that the restrictions of the translations to $F$-classes give rise to a ``section-wise decomposition'' of $\Lambda$ and thus to a description of $\mathcal L$ in parts.

\begin{definition} \label{def:Lambdas}
Let $({\mathcal L}; \leq, \und, 1)$ be a q.n.c.\ tomonoid, let $F$ be a filter of $\mathcal L$, and let $\mathcal P$ be the quotient of $\mathcal L$ by $F$.
Let $R$ be an $F$-class, that is, $R \in {\mathcal P}$. For each $f \in F$, let $\lambda_f^R$ be the restriction of $\lambda_f$ in domain and range to $R$ and put
\[ \Lambda^R \;=\; \{ \lambda_f^R \colon f \in F \}. \]
Moreover, let $T$ be a further $F$-class distinct from $F$, that is, $T \in {\mathcal P} \backslash \{F\}$. Let $S = R \und T$. For any $t \in T$, let $\lambda_t^{R,S} \colon R \to S$ be the mapping arising from $\lambda_t$ by the restriction of its domain to $R$ and of its range to $S$ and put
\[ \Lambda^{R,S} \;=\; \{ \lambda_t^{R,S} \colon t \in T \}. \]
\end{definition}

The sets of mappings specified in Definition \ref{def:Lambdas} will play a major role in what follows. Note that they are well-defined. Indeed, in the notation of Definition \ref{def:Lambdas}, each $F$-class is by construction closed under the multiplication with an element of $F$, that is, $\lambda_f(r) \in R$ for any $r \in R$ and $f \in F$. Moreover, $R \und T = S$ means that $r \und t \in S$ for any $r \in R$ and $t \in T$, that is, $\lambda_t(r) \in S$.

\begin{lemma} \label{lem:coextension-Dreiecke}
Let $({\mathcal L}; \leq, \und, 1)$ be a non-trivial q.n.c.\ tomonoid and let $F$ be a non-trivial filter of $\mathcal L$. Let $\mathcal P$ be the quotient of $\mathcal L$ by $F$ and let $R \in {\mathcal P}$.
Then $(\Lambda^R; \leq, \circ, \id_R)$ is a composition tomonoid fulfilling {\rm (C1)}--{\rm (C4)}; in particular, $\Lambda^R$ is a q.n.c.\ tomonoid. Moreover,
\begin{equation} \label{fml:rho}
\varrho_R \colon F \to \Lambda^R \komma f \mapsto \lambda^R_f
\end{equation}
is a sup-preserving homomorphism from $(F; \leq, \und, 1)$ onto $(\Lambda^R; \leq, \circ, \id_R)$.
\end{lemma}

\begin{proof}
See \cite[Lemma 4.6]{Vet2}.
\end{proof}

Note that, with reference to Lemma \ref{lem:coextension-Dreiecke}, $F$ itself is an $F$-class, in fact $\Lambda^F$ is the Cayley tomonoid of the q.n.c.\ tomonoid $F$ and $\rho_F$ is the isomorphism (\ref{fml:Cayley}) between $F$ and $\Lambda^F$.

Lemma \ref{lem:coextension-Dreiecke} deals with the translations $\lambda_f$ such that $f \in F$. These translations are evidently uniquely determined by the composition tomonoids $\Lambda^R$ together with the homomorphisms $\rho_R$, where $R$ varies over the $F$-classes.

The next Lemma \ref{lem:coextension-Quadrate} deals with the remaining translations, that is, the mappings $\lambda_t$ such that $t \notin F$. They can be piecewise described in a similar fashion. Here, we write $\constant{A}{b}$ for the function that maps all values of a set $A$ to the single value $b$.

\begin{lemma} \label{lem:coextension-Quadrate}
Let $({\mathcal L}; \leq, \und, 1)$ be a non-trivial q.n.c.\ tomonoid and let $F$ be a non-trivial filter of $\mathcal L$. Let $\mathcal P$ be the quotient of $\mathcal L$ by $F$, let $R, T \in {\mathcal P}$ be such that $T < F$, and let $S = R \und T$.
\AufzAnfang
\Nummer{i} Let the pair $R, T \in {\mathcal P}$ be $\und$-maximal. Then $S < R$. Let $b = \inf R$, $\;d = \sup R$, $\;p = \inf S$, and $q = \sup S$. If $b = d$, then $R = \{b\}$, $p \in S$, and $\lambda_t^{R,S}(b) = p$ for all $t \in T$. If $p = q$, then $S = \{p\}$ and $\lambda_t^{R,S} = \constant{R}{p}$ for all $t \in T$.

Assume now $b < d$ and $p < q$. Then $b \in R$ implies $p \in S$. Moreover, $\Lambda^{R,S} = \{ \lambda_t^{R,S} \colon t \in T \}$ is a set of mappings from $R$ to $S$ with the following properties:
\AxiomeAnfang
\Axiom{a} For any $t \in T$, $\lambda_t^{R,S}$ is sup-preserving. If $b \in R$, $\;\lambda^{R,S}_t(b) = p$; if $b \notin R$, $\;\bigwedge_{r \in R} \lambda^{R,S}_t(r) = p$.

\Axiom{b} Under the pointwise order, $\Lambda^{R,S}$ is totally ordered. Moreover, the pointwise calculated supremum of any non-empty set $K \subseteq \Lambda^{R,S}$ is in $\Lambda^{R,S}$, provided that $\bigvee_{\lambda \in K} \lambda(r) \in S$ for all $r \in R$. 

\Axiom{c} If $p \in S$ and $d \in R$, $\;\Lambda^{R,S}$ has the bottom element $\constant{R}{p}$. If $p \in S$ and $d \notin R$, then either $\Lambda^{R,S} = \{ \constant{R}{p} \}$ or $\constant{R}{p} \notin \Lambda^{R,S}$. If $d \notin R$ and $q \in S$, then $p \in S$ and $\Lambda^{R,S} = \{ \constant{R}{p} \}$.
\AxiomeEnde
Moreover, for any $f \in F$ and $t \in T$, we have
\begin{equation} \label{fml:commuting-action}
\lambda^S_f \circ \lambda^{R,S}_t \;=\; \lambda^{R,S}_t \circ \lambda^R_f.
\end{equation}
\Nummer{ii} Let the pair $R, T \in {\mathcal P}$ not be $\und$-maximal. Then $S$ contains a smallest element $p$, and $\lambda^{R,S}_t = \constant{R}{p}$ for all $t \in T$.
\AufzEnde
\end{lemma}

\begin{proof}
See \cite[Lemma 4.7]{Vet2}.
\end{proof}

\section{Archimedean coextensions}
\label{sec:Archimedean-coextensions}

This section is devoted to Archimedean real coextensions. We follow up on our discussion in \cite[Section 6]{Vet2}. We will specify a coextension by describing its Cayley tomonoid ``section-wise'', that is, on the basis of the constituents with which Lemmas \ref{lem:coextension-Dreiecke} and \ref{lem:coextension-Quadrate} deal. We will present explicit formulas for all these constituents.

We will first shortly review the theory, ensuring that a reader without background information can follow the procedure. For further details and unproved facts, we refer to \cite{Vet2}. We carry on determining the different pieces of the newly constructed Cayley tomonoid.

Throughout the section, $\mathcal L$ is a non-trivial q.n.c.\ tomonoid, $F$ is a non-trivial Archimedean filter of $\mathcal L$ such that each $F$-class is order-isomorphic to a real interval, and $\mathcal P$ is the quotient of $\mathcal L$ by $F$. Thus $\mathcal L$ is an Archimedean real coextension of $\mathcal P$ by $F$ and our aim is to describe $\mathcal L$ given $\mathcal P$ and $F$.

We start characterising the extending filter $F$. As we have mentioned already in the introduction, there are, up to isomorphism, only two possibilities. In what follows, the minimum and maximum operation for pairs of real numbers will be denoted by $\wedge$ and $\vee$, respectively.

\begin{definition} \label{def:extending-filter}
Let $\star$ be the \L ukasiewicz t-norm, that is,
\[ \star \colon [0,1]^2 \to [0,1] \komma (a,b) \mapsto (a+b-1) \vee 0; \]
then we call a tomonoid isomorphic to $([0,1]; \leq, \star, 1)$ {\it \L ukasiewicz}.

Furthermore, let $\cdot \colon (0,1]^2 \to (0,1]$ be the usual product of reals; then we call any tomonoid isomorphic to $((0,1]; \leq, \cdot, 1)$ {\it product}.
\end{definition}

\begin{theorem} \label{thm:extending-filter}
If $F$ has a smallest element, $F$ is a \L ukasiewicz tomonoid. If $F$ does not have a smallest element, $F$ is a product tomonoid.
\end{theorem}

\begin{proof}
See, e.g., \cite[Section 2]{KMP}.
\end{proof}

We note that the extending filter can be represented in a way different from Definition \ref{def:extending-filter}. By choosing a suitable base set, the monoidal operation becomes the usual addition of reals.

\begin{lemma} \label{lem:extending-filter-R}
$([-1,0]; \leq, \truncplus, 0)$, where $\truncplus \colon [-1,0]^2 \to [-1,0] \komma (a,b) \mapsto (a+b) \vee -1$, is a \L ukasiewicz tomonoid.

$(\Reals^-; \leq, +, 0)$, where $+$ is the usual addition of reals, is a product tomonoid.
\end{lemma}

\begin{proof}
The isomorphisms are $[0,1] \to [-1,0] \colon a \mapsto a-1$ and $(0,1] \to \Reals^- \colon a \mapsto \ln a$, respectively.
\end{proof}

For the composition tomonoids associated with congruence classes, there are, up to c-isomorphism, only four possibilities.

\begin{definition} \label{def:real-composition-tomonoids}
\AufzAnfang
\Nummer{i} Let $\Phi^\text{\L u}$ consist of the functions $\lambda_t \colon [0,1] \to [0,1] \komma x \mapsto (x+t-1) \vee 0$ for each $t \in [0,1]$. A composition tomonoid c-isomorphic to $(\Phi^\text{\L u}; \leq,$ $\circ, \id_{[0,1]})$ is called {\it \L ukasiewicz}.
\Nummer{ii} Let $\Phi^\text{Pr}$ consist of the functions $\lambda_t \colon (0,1] \to (0,1] \komma x \mapsto t \cdot x$ for each $t \in (0,1]$. A composition tomonoid c-isomorphic to $(\Phi^\text{Pr}; \leq, \circ, \id_{[0,1]})$ is called {\it product}.
\Nummer{iii} Let $\Phi^\text{rP}$ consist of the functions $\lambda_t \colon [0,1) \to [0,1) \komma x \mapsto \frac{(x+t-1) \vee 0}{t}$ for each $t \in (0,1]$. A composition tomonoid c-isomorphic to $(\Phi^\text{rP}; \leq, \circ, \id_{[0,1]})$ is called {\it reversed product}.
\Nummer{iv} Let $\Phi^\text{Po}$ consist of the functions $\lambda_t \colon (0,1) \to (0,1) \komma x \mapsto x^{\frac 1 t}$ for each $t \in (0,1]$. A composition tomonoid c-isomorphic to $(\Phi^\text{Po}; \leq, \circ, \id_{[0,1]})$ is called {\it power}.
\AufzEnde
\end{definition}

The following theorem is the main result of \cite{Vet1}; cf.\ also \cite{Vet2}.

\begin{theorem} \label{thm:Dreiecke}
Let $R \in {\mathcal P}$ and assume that $R$ is not a singleton. If $R$ has a smallest and a largest element, $\Lambda^R$ is \L ukasiewicz. If $R$ has a largest but no smallest element, $\Lambda^R$ is product. If $R$ has a smallest but no largest element, $\Lambda^R$ is reversed product. If $R$ has no smallest and no largest element, $\Lambda^R$ is power.
\end{theorem}

The canonical versions of the four composition tomonoids mentioned in Theorem \ref{thm:Dreiecke} are schematically depicted in Figure \ref{fig:Dreiecksalgebren}.

\begin{figure}[ht!]
\begin{center}
\includegraphics[width=0.98\textwidth]{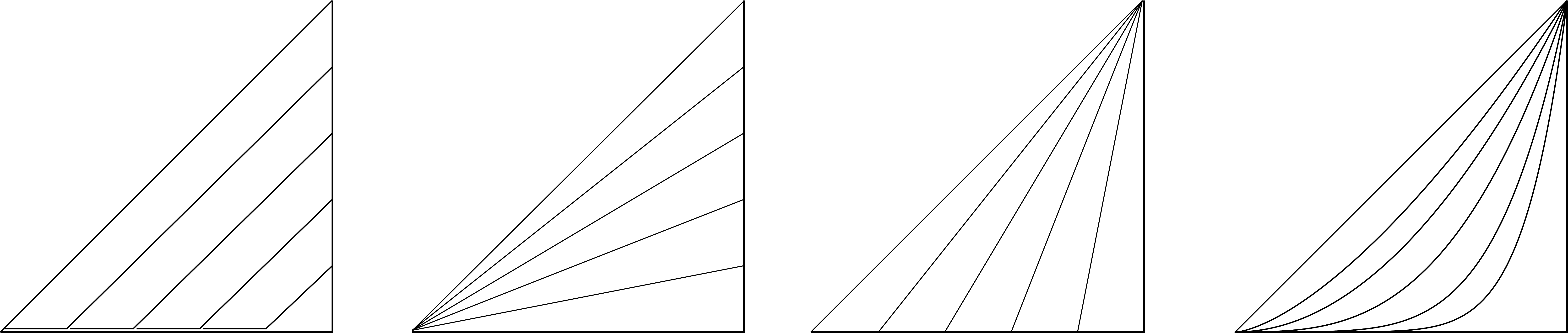}
\caption{The \L ukasiewicz, product, reversed product, and power composition tomonoid, where the base set is, in accordance with Definition \ref{def:real-composition-tomonoids}, the closed, left-open, right-open, open real unit interval, respectively.}
\label{fig:Dreiecksalgebren}
\end{center}
\end{figure}

Again, it turns out that these four composition tomonoids can be represented in a way such that the mappings are of a particularly simple form. To this end, we replace the intervals $[0,1]$, $[0,1)$, $(0,1]$, and $(0,1)$ with $[-1,0]$, $\Reals^+$, $\Reals^-$, and $\Reals$, respectively.

\begin{lemma} \label{lem:composition-tomonoid-R}
The composition tomonoid $(\Lambda^{[-1,0]}; \circ, \leq, \id)$, where $\Lambda^{[-1,0]} = \{ \tau^{[-1,0]}_t \colon$ $t \in [-1,0] \}$ and
\[ \tau^{[-1,0]}_t \colon [-1,0] \to [-1,0] \komma x \mapsto (x+t) \vee -1 \quad\text{for } t \in [-1,0], \]
is \L ukasiewicz. The composition tomonoid $(\Lambda^{\Reals^-}; \circ, \leq, \id)$, where $\Lambda^{\Reals^-} = \{ \tau^{\Reals^-}_t \colon t \in \Reals^- \}$ and
\[ \tau^{\Reals^-}_t \colon \Reals^- \to \Reals^- \komma x \mapsto x+t \quad\text{for } t \in \Reals^-, \]
is product. The composition tomonoid $(\Lambda^{\Reals^+}; \circ, \leq, \id)$, where $\Lambda^{\Reals^+} = \{ \tau^{\Reals^+}_t \colon t \in \Reals^- \}$ and
\[ \tau^{\Reals^+}_t \colon \Reals^+ \to \Reals^+ \komma x \mapsto (x+t) \vee 0 \quad\text{for } t \in \Reals^-, \]
is reversed product. The composition tomonoid $(\Lambda^{\Reals}; \circ, \leq, \id)$, where $\Lambda^{\Reals} = \{ \tau^{\Reals}_t \colon t \in \Reals^- \}$ and
\[ \tau^{\Reals}_t \colon \Reals \to \Reals \komma x \mapsto x+t \quad\text{for } t \in \Reals^-, \]
is power.
\end{lemma}

\begin{proof}
In the first two cases, the isomorphisms are as indicated in the proof of Lemma \ref{lem:extending-filter-R}. In case of $\Lambda^{\Reals^+}$, we use $[0,1) \to \Reals^+ \colon a \mapsto -\ln(1-a)$. In
case of $\Lambda^{\Reals}$, we use $(0,1) \to \Reals \colon a \mapsto -\ln(-\ln a)$.
\end{proof}

We should add that we will use the alternative base sets $[-1,0]$, $\Reals^-$, $\Reals^+$, and $\Reals$, indicated in the Lemmas \ref{lem:extending-filter-R} and \ref{lem:composition-tomonoid-R}, only for the purpose of simplifying the subsequent proofs. When doing so we will say that we use ``auxiliary coordinates''. But the results concerning the tomonoids of Definition \ref{def:extending-filter} and the composition tomonoids of Definition \ref{def:real-composition-tomonoids} will be formulated with reference to the original base sets, which are real unit intervals with or without borders, respectively. In this latter case we will occasionally say that we use ``unit coordinates''.

Recall that $\mathcal P$ is the quotient of the q.n.c.\ tomonoid $\mathcal L$ by the non-trivial filter $F$. We shall now specify, for each $F$-class $R \in {\mathcal P}$, the surjective homomorphisms $\rho_R \colon F \to \Lambda^R$, see (\ref{fml:rho}) in Lemma \ref{lem:coextension-Dreiecke}.

\begin{theorem} \label{thm:filter-homomorphisms}
Let $R \in {\mathcal P}$ be such that $R$ is not a singleton.
\AufzAnfang
\Nummer{i} Assume that $F$ is a \L ukasiewicz tomonoid. Then $\Lambda^R$ is \L ukasiewicz. Furthermore, there is an $\alpha \geq 1$ such that, under the identification of $R$ with the real interval $[0,1]$, the homomorphism $\rho_R \colon F \to \Lambda^R$ is given by
\[ \rho_R(f) \colon [0,1] \to [0,1] \komma r \mapsto (r + \alpha(f-1)) \vee 0,
\quad\text{where } f \in F. \]
\Nummer{ii} Assume that $F$ is a product tomonoid and $\Lambda^R$ is \L ukasiewicz. Then there is an $\alpha > 0$ such that, under the identification of $R$ with $[0,1]$,
\[ \rho_R(f) \colon [0,1] \to [0,1] 
             \komma r \mapsto (r + \alpha \ln f) \vee 0,
\quad\text{where } f \in F. \]
\Nummer{iii} Assume that $F$ is a product tomonoid and $\Lambda^R$ is product. Then there is an $\alpha > 0$ such that, under the identification of $R$ with $(0,1]$,
\[ \rho_R(f) \colon (0,1] \to (0,1] 
             \komma r \mapsto f^\alpha r,
\quad\text{where } f \in F. \]
\Nummer{iv} Assume that $F$ is a product tomonoid and $\Lambda^R$ is reversed product. Then there is an $\alpha > 0$ such that, under the identification of $R$ with $[0,1)$,
\[ \rho_R(f) \colon [0,1) \to [0,1)
             \komma r \mapsto (1 - \tfrac{1-r}{f^\alpha}) \vee 0,
\quad\text{where } f \in F. \]
\Nummer{v} Assume that $F$ is a product tomonoid and $\Lambda^R$ is power. Then there is an $\alpha > 0$ such that, under the identification of $R$ with $(0,1)$,
\[ \rho_R(f) \colon (0,1) \to (0,1)
             \komma r \mapsto r^{\frac 1{f^\alpha}},
\quad\text{where } f \in F. \]
\AufzEnde
\end{theorem}

\begin{proof}
We argue with respect to the auxiliary coordinates. Note first that $\Lambda^{[-1,0]}$ is, as a tomonoid, isomorphic to $([-1,0]; \leq, \truncplus, 0)$, the isomorphism being $[-1,0] \to \Lambda^{[-1,0]} \komma t \mapsto \tau^{[-1,0]}_t$. Similarly, $\Lambda^{\Reals^-}$, $\Lambda^{\Reals^+}$, and $\Lambda^{\Reals}$ are (as tomonoids) isomorphic to $(\Reals^-; \leq, +, 0)$.

(i) In auxiliary coordinates, $F$ is $([-1,0]; \leq, \truncplus, 0)$. Assume that $\Lambda^R$ is \L ukasiewicz, that is, $\Lambda^{[-1,0]}$. We have to determine the surjective homomorphisms $\rho_R$ from $F$ to $\Lambda^R$. Let $\rho$ be a homomorphism from $([-1,0]; \leq, \truncplus, 0)$ to itself. Then either $\rho(r) = 0$ for all $r \in [-1,0]$, or there is an $\alpha \geq 1$ such that $\rho(r) = \alpha r \vee -1$, $\;r \in [-1,0]$. Only in the latter case, $\rho$ is surjective. We conclude
\[ \rho_R \colon [-1,0] \to \Lambda^{[-1,0]}
          \komma f \mapsto \tau^{[-1,0]}_{\alpha f \vee -1}, \]
where $\alpha \geq 1$. Converted to unit coordinates, the assertion follows.

Moreover, the only homomorphism from $([-1,0]; \leq, \truncplus, 0)$ to $(\Reals^-; \leq, +, 0)$ maps all elements to $0$. Hence there is no surjective homomorphism from $F$ to $\Lambda^R$ if $\Lambda^R$ is product, reversed product, or power. Part (i) follows.

(ii) In auxiliary coordinates, $F$ is $(\Reals^-; \leq, +, 0)$ and $\Lambda^R$ is $\Lambda^{[-1,0]}$. Let $\rho$ be a homomorphism from $(\Reals^-; \leq, +, 0)$ to $([-1,0]; \leq, \truncplus, 0)$. Then there is an $\alpha \geq 0$ such that $\rho(r) = \alpha r \vee -1$, $\;r \in \Reals^-$. Exactly in case $\alpha > 0$, $\rho$ is surjective. Hence, for some $\alpha > 0$, we have
\[ \rho_F \colon \Reals^- \to \Lambda^{[-1,0]}
          \komma \tau^{[-1,0]}_{\alpha f \vee -1} \]
and the assertion follows.

(iii)--(v) In all these cases, $\Lambda^R$ is, as a tomonoid, is isomorphic to $(\Reals^-; \leq, +, 0)$. Let $\rho$ be a surjective homomorphism from $(\Reals^-; \leq, +, 0)$ to itself. Then again, there is an $\alpha > 0$ such that $\rho(r) = \alpha r$, $\;r \in \Reals^-$. Hence, for some $\alpha > 0$, we have
\[ \rho_F \colon \Reals^- \to \Lambda^R
          \komma \tau^R_{\alpha f}, \]
where $R$ is $\Reals^-$, $\Reals^+$, or $\Reals$, respectively. The remaining assertions follow as well.
\end{proof}

Theorems \ref{thm:Dreiecke} and \ref{thm:filter-homomorphisms} describe the upper-most part of the Cayley tomonoid of the coextension $\mathcal L$: the translations by the elements of the filter $F$. Let us now turn to the remaining translations. As we have observed in \cite{Vet2}, for each pair of congruence classes $R$ and $S$, the sets $\Lambda^{R,S}$ are to a large extent determined by $\Lambda^R$ and $\Lambda^S$. Our aim is to explicitly determine $\Lambda^{R,S}$; to this end, we have to consider all 16 combinations of c-isomorphism types of $\Lambda^R$ and $\Lambda^S$. Not all combinations are possible, however, and in some further cases $\Lambda^{R,S}$ contains only a single constant mapping. We will see that nine non-trivial possibilities remain.

We first determine the five impossible cases.

\begin{proposition}
Let $R, T \in {\mathcal P}$, where $T < F$, be $\und$-maximal, let $S = R \und T$, and assume that neither $R$ nor $S$ is a singleton.
\AufzAnfang
\Nummer{i} If $\Lambda^R$ is \L ukasiewicz, then $\Lambda^S$ is \L ukasiewicz or reversed product.

\Nummer{ii} If $\Lambda^R$ is reversed product, then $\Lambda^S$ is \L ukasiewicz or reversed product.

\Nummer{iii} If $\Lambda^R$ is power, then $\Lambda^S$ is \L ukasiewicz, reversed product, or power.
\AufzEnde
\end{proposition}

\begin{proof}
(i)--(ii) If $\Lambda^R$ is \L ukasiewicz or reversed product, $R$ has a smallest element. Hence, by Lemma \ref{lem:coextension-Quadrate}(i), also $S$ has a smallest element, that is, also $\Lambda^S$ is \L ukasiewicz or reversed product.

(iii) If $\Lambda^R$ is power, then $R$ has no largest element. If, in addition, $S$ has a largest element, then $S$ has by Lemma \ref{lem:coextension-Quadrate}(i)(c) also a smallest element. Hence $\Lambda^S$ cannot be product.
\end{proof}

For two congruence classes $R$ and $S$, call $\Lambda^{R,S}$ {\it trivial} if $S$ has a bottom element $p$ and $\Lambda^{R,S}$ consists of the single element $\constant{R}{p}$.

\begin{proposition} \label{prop:trivial-LambdaRS}
Let $R, T \in {\mathcal P}$ be such that $T < F$ and let $S = R \und T$. Then $\Lambda^{R,S}$ is trivial in each of the following cases:
\AufzAnfang
\Nummer{i} The pair $R, T$ is not $\und$-maximal.

\Nummer{ii} $R$ is a singleton.

\Nummer{iii} $S$ is a singleton.

\Nummer{iv} $\Lambda^R$ is reversed product and $\Lambda^S$ is \L ukasiewicz.

\Nummer{v} $\Lambda^R$ is power and $\Lambda^S$ is \L ukasiewicz.
\AufzEnde
\end{proposition}

\begin{proof}
(i) holds by Lemma \ref{lem:coextension-Quadrate}(ii).

(ii) holds by Lemma \ref{lem:coextension-Quadrate}(i),(ii).

(iii) is obvious.

(iv)--(v) By parts (i)--(iii), it suffices to show the claims in the case that $R, T$ is a $\und$-maximal pair and neither $R$ nor $S$ is a singleton. If $\Lambda^R$ is reversed product or power and $\Lambda^S$ is \L ukasiewicz, then $R$ has no largest element and $S$ has a largest element. By Lemma \ref{lem:coextension-Quadrate}(i)(c) it follows that $\Lambda^{R,S}$ is trivial.
\end{proof}

In the remaining cases, we determine $\Lambda^{R,S}$ on the basis of the following lemma, for whose proof we refer to \cite{Vet2}.

\begin{lemma} \label{lem:Xi}
Let $R, T \in {\mathcal P}$, where $T < F$, be $\und$-maximal, let $S = R \und T$, and assume that neither $R$ nor $S$ is a singleton. Then $\Lambda^{R,S}$ is contained in the set
\begin{equation} \label{fml:Xi}
\Xi^{R,S} \;=\; \{ \xi \colon R \to S \;\colon\;
             \xi \circ \rho_R(f) = \rho_S(f) \circ \xi \;\text{ \rm for all $f \in F$} \}.
\end{equation}
In fact, if $R$ has a largest element or $S$ has no smallest element, then either $\Lambda^{R,S} = \Xi^{R,S}$ or there is a $\zeta \in \Xi^{R,S}$ such that $\Lambda^{R,S} = \{ \xi \in \Xi^{R,S} \colon \xi \leq \zeta \}$. If $S$ has a smallest element $p$ but neither $R$ nor $S$ have a largest element, then either $\Lambda^{R,S}$ is trivial or $\Lambda^{R,S} = \Xi^{R,S} \ohne \{ \constant{R}{p} \}$ or there is a $\zeta \in \Xi^{R,S} \ohne \{ \constant{R}{p} \}$ such that $\Lambda^{R,S} = \{ \xi \in \Xi^{R,S} \ohne \{ \constant{R}{p} \} \colon \xi \leq \zeta \}$.
\end{lemma}

Given a pair of congruence classes $R$ and $S$, where $S = R \und T$ for some $T < F$, Lemma \ref{lem:Xi} states that $\Lambda^{R,S}$ is essentially uniquely determined by $\Lambda^R$ and $\Lambda^S$ together with the homomorphisms $\rho_R$ and $\rho_S$. Indeed, knowing $\lambda^R_f$ and $\lambda^S_f$ for each $f \in F$ means that we can say the following: for any non-minimal element $r$ of $R$ and any non-minimal element $s$ of $S$, there is at most one mapping $\lambda \colon R \to S$ in $\Lambda^{R,S}$ such that $\lambda(r) = s$. An illustration of this fact can be found in Figure \ref{fig:Xi}.

\begin{figure}[ht]
\begin{center}
\includegraphics[width=0.6\textwidth]{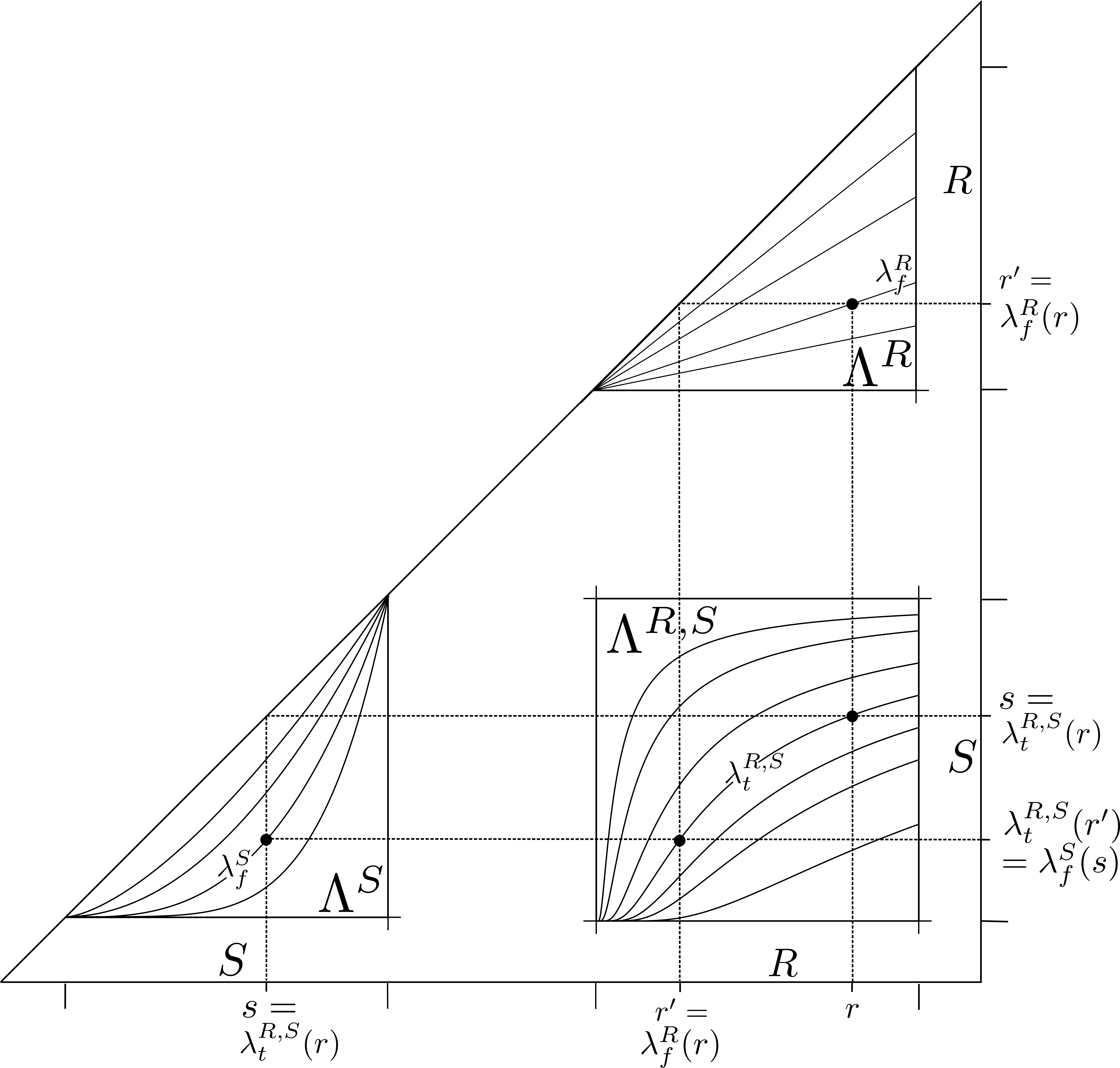}
\caption{Assume that, for some $t \in T$, we have $\lambda^{R,S}_t(r) = s$. Let $f \in F$ and $r' = \lambda^R_f(r)$. Now, applying to $r$ first $\lambda^R_f$ and then $\lambda^{R,S}_t$ leads to the same result as applying first $\lambda^{R,S}_t$ and then $\lambda^S_f$. That is, $\lambda^{R,S}(r') = \lambda^{R,S}_t(\lambda^R_f(r)) = \lambda^S_f(\lambda^{R,S}_t(r)) = \lambda^S_f(s)$ and we see that the value of $\lambda^{R,S}$ at $r'$ depends only on $s$ and $\lambda^S_f$. It is in fact possible to show that the whole mapping $\lambda^{R,S}_t$ is uniquely determined by its value at a single point.
}
\label{fig:Xi}
\end{center}
\end{figure}

We conclude from Lemma \ref{lem:Xi} that in order to determine $\Lambda^{R,S}$ we have to calculate $\Xi^{R,S}$ as defined by (\ref{fml:Xi}). In what follows, we will do so for every possible combination of c-isomorphism types of $\Lambda^R$ and $\Lambda^S$.

In the theorems that follow, we again use unit coordinates and we assume that the homomorphisms $\rho_R \colon F \to \Lambda^R$ and $\rho_S \colon F \to \Lambda^S$ are given according to Theorem \ref{thm:filter-homomorphisms}. We will denote the coefficient applying to $\rho_R$ by $\alpha_R$ and the one applying to $\rho_S$ by $\alpha_S$.

\begin{theorem} \label{thm:LambdaRS-Lukasiewicz}
Let the filter $F$ have a smallest element. Let $R, T \in {\mathcal P}$, where $T < F$, be $\und$-maximal, let $S = R \und T$, and assume that neither $R$ nor $S$ is a singleton. Then $\Lambda^R$ and $\Lambda^S$ are \L ukasiewicz and, under the identification of $R$ and $S$ with $[0,1]$, $\;\Lambda^{R,S}$ consists of the mappings
\[ \lambda_z \colon [0,1] \to [0,1]
             \komma r \mapsto (\tfrac{\alpha_S}{\alpha_R} r + z) \vee 0, \]
where $-\frac{\alpha_S}{\alpha_R} \leq z \leq m$ for some $m \in [-\frac{\alpha_S}{\alpha_R}, (1-\frac{\alpha_S}{\alpha_R}) \wedge 0]$.
\end{theorem}

\begin{proof}
In auxiliary coordinates, the filter is $([-1,0]; \leq, \truncplus, 0)$, the composition tomonoids are $\{ \tau^{[-1,0]}_t \colon t \in [-1,0]\}$ and $\{ \tau^{[-1,0]}_t \colon t \in [-1,0]\}$, and the homomorphisms are $f \mapsto \tau^{[-1,0]}_{\alpha_R f \vee -1}$ and $f \mapsto \tau^{[-1,0]}_{\alpha_S f \vee -1}$, respectively. We have to determine the mappings $\lambda \colon [-1,0] \to [-1,0]$ such that
\begin{equation} \label{fml:Luk-LukLuk}
\lambda \circ \tau^{[-1,0]}_{\alpha_R f \vee -1} \;=\; \tau^{[-1,0]}_{\alpha_S f \vee -1} \circ \lambda,
\quad f \in [-1,0].
\end{equation}
This means $\lambda((r + \alpha_R f) \vee -1) = (\lambda(r) + \alpha_S f) \vee -1$ for all $f, r \in [-1,0]$, or $\lambda((r + d) \vee -1) = (\lambda(r) + \frac{\alpha_S}{\alpha_R} d) \vee -1$ for all $r \in [-1,0]$ and $d \in [-\alpha_R, 0]$. Putting $l = \lambda(0)$, we conclude that the solutions are of the form
\begin{equation} \label{fml:real-lambda}
\lambda(r) = (\tfrac{\alpha_S}{\alpha_R} r + l) \vee -1, \quad r \in [-1,0].
\end{equation}
If $\alpha_S \geq \alpha_R$, (\ref{fml:Luk-LukLuk}) is fulfilled for all $l \in [-1,0]$; if $\alpha_S < \alpha_R$, (\ref{fml:Luk-LukLuk}) is fulfilled if and only if $l \leq \frac{\alpha_S}{\alpha_R} -1$. A conversion to unit coordinates proves the assertion.
\end{proof}

\begin{theorem} \label{thm:LambdaRS-product}
Let the filter $F$ not have a smallest element. Let $R, T \in {\mathcal P}$, where $T < F$, be $\und$-maximal, and let $S = R \und T$.
\AufzAnfang
\Nummer{i} Let $\Lambda^R$ and $\Lambda^S$ be \L ukasiewicz. Then, under the identification of $R$ and $S$ with $[0,1]$, $\;\Lambda^{R,S}$ consists of the mappings
\[ \lambda_z \colon [0,1] \to [0,1]
             \komma r \mapsto (\tfrac{\alpha_S}{\alpha_R} r + z) \vee 0, \]
where $-\frac{\alpha_S}{\alpha_R} \leq z \leq m$ for some $m \in [-\frac{\alpha_S}{\alpha_R}, (1-\frac{\alpha_S}{\alpha_R}) \wedge 0]$.

\Nummer{ii} Let $\Lambda^R$ be \L ukasiewicz and $\Lambda^S$ reversed product. Then, under the identification of $R$ and $S$ with $[0,1]$ and $[0,1)$, respectively, $\Lambda^{R,S}$ consists of the mappings
\[ \lambda_z \colon [0,1] \to [0,1) 
             \komma r \mapsto (1 - e^{- \frac{\alpha_S}{\alpha_R} (r+z)}) \vee 0, \]
where $-1 \leq z \leq m$ for some $m \in [-1,0]$.

\Nummer{iii} Let $\Lambda^R$ be product and $\Lambda^S$ \L ukasiewicz. Then, under the identification of $R$ and $S$ with $(0,1]$ and $[0,1]$, respectively, $\Lambda^{R,S}$ consists of the mappings
\[ \lambda_z \colon (0,1] \to [0,1]
             \komma r \mapsto (\tfrac{\alpha_S}{\alpha_R} \ln r + z) \vee 0, \]
where $0 \leq z \leq m$ for some $m \in [0,1]$.

\Nummer{iv} Let $\Lambda^R$ and $\Lambda^S$ be product. Then, under the identification of $R$ and $S$ with $(0,1]$, $\;\Lambda^{R,S}$ consists of the mappings
\[ \lambda_z \colon (0,1] \to (0,1] 
             \komma r \mapsto z \, r^{\frac{\alpha_S}{\alpha_R}}, \]
where $0 < z \leq m$ for some $m \in (0,1]$.

\Nummer{v} Let $\Lambda^R$ be product and $\Lambda^S$ reversed product. Then, under the identification of $R$ and $S$ with $(0,1]$ and $[0,1)$, respectively, $\Lambda^{R,S}$ consists of the mappings
\[ \lambda_z \colon (0,1] \to [0,1) 
             \komma r \mapsto (1 - \frac 1{z \, r^{\frac{\alpha_S}{\alpha_R}}}) \vee 0, \]
where either $z \geq 1$ or $1 \leq z \leq m$ for some $m \geq 1$.

\Nummer{vi} Let $\Lambda^R$ be product and $\Lambda^S$ power. Then, under the identification of $R$ and $S$ with $(0,1]$ and $(0,1)$, respectively, $\Lambda^{R,S}$ consists of the mappings
\[ \lambda_z \colon (0,1] \to (0,1) 
             \komma r \mapsto z^{r^{-\frac{\alpha_S}{\alpha_R}}}, \]
where either $0 < z < 1$ or $0 < z \leq m$ for some $m \in (0,1)$.

\Nummer{vii} Let $\Lambda^R$ and $\Lambda^S$ be reversed product. Then, under the identification of $R$ and $S$ with $[0,1)$, $\;\Lambda^{R,S}$ is either trivial or contained in the set of mappings
\[ \lambda_z \colon [0,1) \to [0,1) 
             \komma r \mapsto (1 - \frac{(1-r)^{\frac{\alpha_S}{\alpha_R}}}z) \vee 0, \]
where $0 < z \leq m$ for some $m \in (0,1]$.

\Nummer{viii} Let $\Lambda^R$ be power and $\Lambda^S$ reversed product. Then, under the identification of $R$ and $S$ with $(0,1)$ and $[0,1)$, respectively, $\Lambda^{R,S}$ is either trivial or contained in the set of mappings
\[ \lambda_z \colon (0,1) \to [0,1)
             \komma r \mapsto (1 - \frac{(-\ln r)^{\frac{\alpha_S}{\alpha_R}}}z) \vee 0, \]
where either $z > 0$ or $0 < z \leq m$ for some $m > 0$.

\Nummer{ix} Let $\Lambda^R$ and $\Lambda^S$ be power. Then, under the identification of $R$ and $S$ with $(0,1)$, $\;\Lambda^{R,S}$ consists of the mappings
\[ \lambda_z \colon (0,1) \to (0,1)
             \komma r \mapsto z^{(-\ln r)^{\frac{\alpha_S}{\alpha_R}}}, \]
where either $0 < z < 1$ or $0 < z \leq m$ for some $m \in (0,1)$.
\AufzEnde
\end{theorem}

\begin{proof}
(i) In auxiliary coordinates, the filter is $(\Reals^-; \leq, +, 0)$. Apart from that, the situation is as in Theorem \ref{thm:LambdaRS-Lukasiewicz}; we proceed similarly as in the proof of that theorem.

(ii) The filter is $(\Reals^-; \leq, +, 0)$; the composition tomonoids are $\{ \tau^{[-1,0]}_t \colon t \in [-1,0] \}$ and $\{ \tau^{\Reals^+}_t \colon t \in \Reals^- \}$, and the homomorphisms are $f \mapsto \tau^{[-1,0]}_{\alpha_R f \vee -1}$ and $f \mapsto \tau^{\Reals^+}_{\alpha_S f}$, respectively. Hence we have to determine $\lambda \colon [-1,0] \to \Reals^+$ such that
\begin{equation} \label{fml:Pro-LukRp}
\lambda \circ \tau^{[-1,0]}_{\alpha_R f \vee -1} \;=\; \tau^{\Reals^+}_{\alpha_S f} \circ \lambda, \quad f \in \Reals^-.
\end{equation}
This means $\lambda((r + \alpha_R f) \vee -1) = (\lambda(r) + \alpha_S f) \vee 0$ for all $f \leq 0$ and $r \in [-1,0]$. Putting $l = \lambda(0)$, we get $\lambda(r) = (\tfrac{\alpha_S}{\alpha_R} r + l) \vee 0$, and (\ref{fml:Pro-LukRp}) is fulfilled if and only if $l \leq \frac{\alpha_S}{\alpha_R}$.

(iii) We have to determine $\lambda \colon \Reals^- \to [-1,0]$ such that
\begin{equation} \label{fml:Pro-ProLuk}
\lambda \circ \tau^{\Reals^-}_{\alpha_R f} \;=\; \tau^{[-1,0]}_{\alpha_S f \vee -1} \circ \lambda, \quad f \in \Reals^-.
\end{equation}
It follows $\lambda(r + \alpha_R f) = (\lambda(r) + \alpha_S f) \vee -1$ for $f, r \leq 0$. Putting $l = \lambda(0)$, we get $\lambda(r) = (\tfrac{\alpha_S}{\alpha_R} r + l) \vee -1$, and these mappings are solutions for all $l \in [-1,0]$.

(iv) We have to determine $\lambda \colon \Reals^- \to \Reals^-$ such that
\[ \lambda \circ \tau^{\Reals^-}_{\alpha_R f} \;=\; \tau^{\Reals^-}_{\alpha_S f} \circ \lambda, \quad f \in \Reals^-. \]
The solutions are $\lambda(r) = \tfrac{\alpha_S}{\alpha_R} r + l$, where $l \in \Reals^-$.

(v) We have to determine $\lambda \colon \Reals^- \to \Reals^+$ such that
\[ \lambda \circ \tau^{\Reals^-}_{\alpha_R f} \;=\; \tau^{\Reals^+}_{\alpha_S f} \circ \lambda, \quad f \in \Reals^-. \]
The solutions are $\lambda(r) = (\tfrac{\alpha_S}{\alpha_R} r + l) \vee 0$, where $l \in \Reals^+$.

(vi) We have to determine $\lambda \colon \Reals^- \to \Reals$ such that
\[ \lambda \circ \tau^{\Reals^-}_{\alpha_R f} \;=\; \tau^{\Reals}_{\alpha_S f} \circ \lambda, \quad f \in \Reals^-. \]
We conclude $\lambda(r) = \frac{\alpha_S}{\alpha_R} r + l$, where $l \in \Reals$.

(vii) We have to determine $\lambda \colon \Reals^+ \to \Reals^+$ such that
\[ \lambda \circ \tau^{\Reals^+}_{\alpha_R f} \;=\; \tau^{\Reals^+}_{\alpha_S f} \circ \lambda, \quad f \in \Reals^-. \]
This means $\lambda((r + \alpha_R f) \vee 0) = (\lambda(r) + \alpha_S f) \vee 0$ for any $f \leq 0$ and $r \geq 0$. One solution is $\lambda(r) = 0$ for all $r$. Otherwise, there is a $C > 0$ such that $\lambda(C) > 0$. For $0 \leq r \leq C$ we then have $\lambda(r) = (\frac{\alpha_S}{\alpha_R} r + l) \vee 0$ for some $l \leq 0$. As $\lambda$ is monotone, we can choose $C$ arbitrarily large; we conclude that $\lambda$ is for all $r$ of this form.

(viii) We have to determine $\lambda \colon \Reals \to \Reals^+$ such that
\[ \lambda \circ \tau^{\Reals}_{\alpha_R f} \;=\; \tau^{\Reals^+}_{\alpha_S f} \circ \lambda, \quad f \in \Reals^-. \]
This means $\lambda(r + \alpha_R f) = (\lambda(r) + \alpha_S f) \vee 0$ for any $f \leq 0$ and any $r$. Again, one solution is $\lambda(r) = 0$ for all $r$. Otherwise, there is a $C \in \Reals$ such that $\lambda(C) > 0$ and we conclude similarly as in part (vii) that the solutions are $\lambda(r) = (\frac{\alpha_S}{\alpha_R} r + l) \vee 0$ for any $l \in \Reals$.

(ix) We have to determine $\lambda \colon \Reals \to \Reals$ such that
\[ \lambda \circ \tau^{\Reals}_{\alpha_R f} \;=\; \tau^{\Reals}_{\alpha_S f} \circ \lambda, \quad f \in \Reals^-. \]
This means $\lambda(r + \alpha_R f) = \lambda(r) + \alpha_S f$ for any $f \leq 0$ and any $r$. The solutions are $\lambda(r) = \frac{\alpha_S}{\alpha_R} r + l$ for any $l \in \Reals$.
\end{proof}

A graphical illustration of Theorem \ref{thm:LambdaRS-product} is found in Figure \ref{fig:Quadratalgebren}.

\begin{figure}[ht!]
\begin{center}
\includegraphics[width=0.8\textwidth]{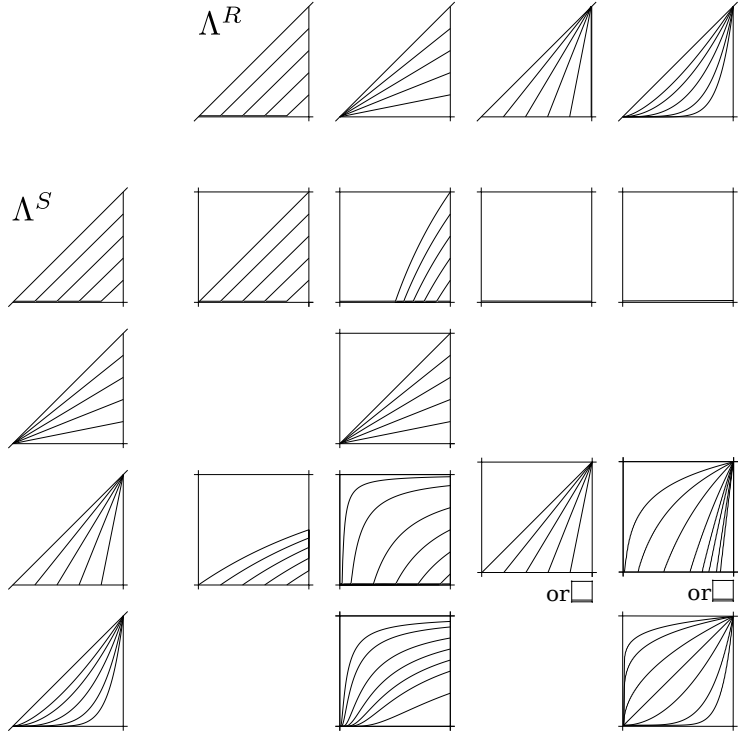}
\caption{A qualitative view on the sets $\Lambda^{R,S}$ depending on $\Lambda^R$ and $\Lambda^S$, in the context of an Archimedean real coextension. Standard coordinates are used and it is assumed that $\alpha_R = \alpha_S = 1$. $\Lambda^{R,S}$ is a downwards closed subset of the indicated set of mappings, respectively.}
\label{fig:Quadratalgebren}
\end{center}
\end{figure}

\section{Semilattice coextensions}
\label{sec:semilattice-coextensions}

In this section, we investigate semilattice real coextensions. The extending filter is in this case a semilattice and can hence be identified with an almost complete chain; we recall that the monoidal product is simply the minimum. Whereas an Archimedean q.n.c.\ tomonoid $F$ possesses by definition no filter apart from $\{1\}$ and $F$, a semilattice $F$ possesses so-to-say the maximal possible amount of filters, namely, $d^\leq$ for each $d \in F$ as well as $d^<$ for each $d \in \mitNull{F} \ohne \{1\}$ such that $d = \bigwedge_{a > d} a$. Consequently, whereas an Archimedean coextension cannot be split into two successive non-trivial coextensions, a non-trivial semilattice real coextension can be replaced by a sequence of non-trivial coextensions whose length is not bounded from above.

We proceed in analogy to the previous section. Again, $\mathcal L$ is supposed to be a non-trivial q.n.c.\ tomonoid, $F$ is a non-trivial semilattice filter such that each $F$-class is order-isomorphic to a real interval, and $\mathcal P$ is the quotient of $\mathcal L$ by $F$. We wish to characterise $\mathcal L$ given $\mathcal P$ and $F$.

We start with a simple observation.

\begin{lemma}
Each $f \in F$ is an idempotent element of $\mathcal L$.
\end{lemma}

%
%

We conclude from Proposition \ref{prop:tomonoid-composition-tomonoid} that the translations $\lambda_f$, where $f$ belongs to the filter $F$, are mappings such that $\lambda_f \circ \lambda_f = \lambda_f$. We will call mappings of a set to itself with this property {\it idempotent} as well. We may describe the idempotent translations by the set of their fixpoints.

\begin{proposition} \label{prop:idempotent}
Let $({\mathcal L}; \leq)$ be an almost complete chain. Let $\lambda \colon {\mathcal L} \to {\mathcal L}$ be sup-preserving, shrinking, and idempotent. Let
\[ \fp{\lambda} \;=\; \{ a \in {\mathcal L} \colon \lambda(a) = a \}. \]
Then $\fp{\lambda}$ is a set $E \subseteq {\mathcal L}$ with the following properties:
\AxiomeAnfang
\Axiom{E1} $E$ is closed under suprema.

\Axiom{E2} Let $a \in E$ be such that $a = \sup\, \{ x \in {\mathcal L} \colon x < a \}$. Then $a = \sup\, \{ x \in E \colon x < a \}$.

\Axiom{E3} For each $a \in {\mathcal L}$ there is an $e \in E$ such that $e \leq a$.
\AxiomeEnde
Moreover, we have
\begin{equation} \label{fml:idempotent-from-fp}
\lambda(a) \;=\; \max \, \{ e \in \fp{\lambda} \colon e \leq a \}
\end{equation}
for any $a \in {\mathcal L}$. Conversely, let $E \subseteq {\mathcal L}$ fulfil {\rm (E1)}--{\rm (E3)}. Then the mapping
\begin{equation} \label{fml:idp-aus-fp}
\lambda \colon {\mathcal L} \to {\mathcal L} \komma a 
           \mapsto \max \, \{ e \in E \colon e \leq a \}
\end{equation}
is sup-preserving, shrinking, and idempotent. Moreover, $\fp{\lambda} = E$.
\end{proposition}

\begin{proof}
Let $\lambda \colon {\mathcal L} \to {\mathcal L}$ be sup-preserving, shrinking, and idempotent. Let us first note that $\fp{\lambda}$ is exactly the image of $\lambda$. Indeed, each element of $\fp{\lambda}$ is of the form $\lambda(a)$ for some $a \in {\mathcal L}$. Conversely,
\begin{equation} \label{fml:image-in-fp}
\lambda(a) \in \fp{\lambda}
\end{equation}
for any $a \in {\mathcal L}$ because, in view of the idempotency of $\lambda$, we have $\lambda(\lambda(a)) = \lambda(a)$.

Assume now that $\lambda(a_\iota) = a_\iota$ for $a_\iota \in {\mathcal L}$, $\iota \in I$; then $\lambda(\bigvee_\iota a_\iota) = \bigvee_\iota \lambda(a_\iota) = \bigvee_\iota a_\iota$ and hence $\fp{\lambda}$ fulfils (E1). Next, let $a \in \fp{\lambda}$ be such that $a = \sup\, \{ x \in {\mathcal L} \colon x < a \}$. Assume that there is a $b < a$ in $\mathcal L$ such that $\fp{\lambda}$ does not contain any element $x$ such that $b < x < a$. Then for all $x$ such that $b < x < a$ we have $\lambda(x) \leq b$. Indeed, $\lambda(x) \leq x$ because $\lambda$ is shrinking, and $\lambda(x) \in \fp{\lambda}$ by (\ref{fml:image-in-fp}). Since $\lambda$ is sup-preserving, also $\lambda(a) \leq b < a$ holds and hence $a \notin \fp{\lambda}$. It follows that $\fp{\lambda}$ fulfils (E2). Finally, for $a \in {\mathcal L}$, $\lambda(a) \leq a$ is in $\fp{\lambda}$ and hence $\fp{\lambda}$ fulfils also (E3).

To show (\ref{fml:idempotent-from-fp}), let $a \in {\mathcal L}$. Again, $\lambda(a) \leq a$ because $\lambda$ is shrinking, and $\lambda(a) \in \fp{\lambda}$ by (\ref{fml:image-in-fp}). Let $e$ be the largest element with these two properties, that is, the largest element of $\fp{\lambda}$ less or equal to $a$. By (E1), such an element exists. We have to show $\lambda(a) = e$. Indeed, on the one hand $\lambda(a) \leq e$. On the other hand, $e \leq a$ implies $e = \lambda(e) \leq \lambda(a)$. Hence $\lambda(a) = e$ and the proof of the first half of the proposition is complete.

For the second half, let $E \subseteq {\mathcal L}$ fulfil (E1)--(E3). Then we can define $\lambda$ by (\ref{fml:idp-aus-fp}) thanks to (E3) and (E1). Clearly, $\lambda$ is shrinking and idempotent. It is also clear that $e \in E$ if and only if $\lambda(e) = e$, that is, if and only if $e \in \fp{\lambda}$. It is furthermore obvious that $\lambda$ is monotone. It remains to show that $\lambda$ is sup-preserving. Let $a_\iota \in {\mathcal L}$, $\iota \in I$, let $a = \bigvee_\iota a_\iota$, and assume $a_\iota < a$ for each $\iota$. Then either $a \in E$ and hence $\lambda(a) = a$. In this case, (E2) implies that $a = \sup\, \{ x \in E \colon x < a \}$ and we conclude that also $\bigvee_\iota \lambda(a_\iota) = a$. Or $a \notin E$; let then $b = \lambda(a)$. In this case, we have $b < a$ and we conclude that also $\bigvee_\iota \lambda(a_\iota) = b$.
\end{proof}

For a non-empty subset $A$ of an almost complete chain $R$, we denote by $A^\vee$ the set of all suprema of non-empty subsets of $A$.

\begin{theorem} \label{thm:mu}
Let $\mathcal E$ be the set of all subsets $E$ of $\mathcal L$ fulfilling {\rm (E1)}--{\rm (E3)}. Let
\begin{equation} \label{fml:mu}
\mu \colon F \to {\mathcal E} \komma f \mapsto \fp{\lambda_f}.
\end{equation}
Then $\mu$ is monotone; indeed we have $\mu(\bigvee_\iota f_\iota) = (\bigcup_\iota \mu(f_\iota))^\vee$ and $\mu(1) = {\mathcal L}$.
\end{theorem}

\begin{proof}
From the fact that each translation is shrinking we conclude that $\mu$ is monotone. Moreover, clearly $\mu(1) = {\mathcal L}$.

Let $\lambda_\iota$, $\iota \in I$, be translations of $\mathcal L$, and let $\lambda = \bigvee_\iota \lambda_\iota$. Then $\fp{\lambda_\iota} \subseteq \fp{\lambda}$ for each $\iota$ and, since $\fp{\lambda}$ is closed under suprema, $(\bigcup_\iota \fp{\lambda_\iota})^\vee \subseteq \fp{\lambda}$. Conversely, let $a \in \fp{\lambda}$. Then we have $\bigvee_\iota \lambda_\iota(a) = \lambda(a) = a$. Since $\lambda_\iota(a) \in \fp{\lambda_\iota}$ for any $\iota$, we conclude that $a \in (\bigcup_\iota \fp{\lambda_\iota})^\vee$.
\end{proof}

We summarise that the translations by the elements of the filter $F$ are fully specified by the map $\mu$ in (\ref{fml:mu}), which maps $F$ into the set $\mathcal E$ of sets fulfilling the conditions (E1)--(E3).

The map $\mu$, however, can be complicated. We do not further discuss the general case; we rather make a simplifying assumption. Let us define the following condition:
\AufzAnfang
\Nummer{G} For each $F$-class $R$ that is not a singleton, there is a bijection $\nu_R \colon F \to R$ such that one of the following possibilities applies: either $\nu_R$ is order-preserving and, for any $f \in F$,
\[ \fp{\lambda_f} \cap R \;=\; \{ r \in R \colon r \leq \nu_R(f) \} \]
or $\nu_R$ is order-reversing and, for any $f \in F$,
\[ \fp{\lambda_f} \cap R
   \;=\; \{ r \in R \colon r > \nu_R(f) \text{ or $r$ is the smallest element of $R$} \}. \]
\AufzEnde

We note that we have considered in \cite{Vet1} a condition similar to (G), in order to define the so-called regular l.-c.\ t-norms.

\begin{definition} \label{def:Goedel-composition-tomonoids}
\AufzAnfang
\Nummer{i} Let $(R; \leq)$ be a chain. Let $\Phi^\text{G\" o}$ consist of the functions $\lambda_t \colon R \to R \komma x \mapsto x \wedge t$ for each $t \in R$. Then the composition tomonoid $(\Phi^\text{G\" o}; \leq, \circ, \id_R)$ is called {\it G\" odel}.
\Nummer{ii} Let $(R; \leq)$ be a chain with a bottom element $0$. Let $\Phi^\text{rG}$ consist of the functions
\[ \lambda_t \colon R \to R \komma x \mapsto
   \begin{cases} 0 & \text{if $x \leq t$,} \\
                 x & \text{if $x > t$}
   \end{cases} \]
for each $t \in R$. Then the composition tomonoid $(\Phi^\text{rG}; \leq, \circ, \id_R)$ is called {\it reversed G\" odel}.
\AufzEnde
\end{definition}

We assume from now on that condition (G) is fulfilled. 

\begin{theorem} \label{thm:Goedel-Dreiecke}
Let $R \in {\mathcal P}$ not be a singleton. Then $\Lambda^R$ is either a G\" odel or a reversed G\" odel composition tomonoid.
\end{theorem}

\begin{proof}
By Proposition \ref{prop:idempotent}, we have for each $f \in F$ and $r \in R$ that $\lambda_f(r) = \max \, \{ e \in \fp{\lambda_f} \colon e \leq r \}$.

Let $\nu_R \colon F \to R$ be the bijection according to (G). Assume first that $\nu_R$ is order-preserving. Then $\fp{\lambda_f} \cap R = \{ r \in R \colon r \leq \nu_R(f) \}$. Hence $\lambda_f(r) = r$ if $r \leq \nu_R(f)$, and $\lambda_f(r) = \nu_R(f)$ otherwise. We conclude that $\Lambda^R$ is a G\" odel composition tomonoid.

Assume now that $\nu_R$ is order-reversing. Then $R$ has a smallest element $b$. For $f \in F$, we then have $\fp{\lambda_f} \cap R = \{ r \in R \colon r > \nu_R(f) \} \cup \{b\}$. Hence $\lambda_f(r) = b$ if $r \leq \nu_R(r)$, and $\lambda_f(r) = r$ otherwise. Thus $\Lambda^R$ is a reversed G\" odel composition tomonoid.
\end{proof}

The G\" odel and the reversed G\" odel composition tomonoid are schematically depicted in Figure \ref{fig:Goedel-Dreiecksalgebren}.

\begin{figure}[ht!]
\begin{center}
\includegraphics[width=0.5\textwidth]{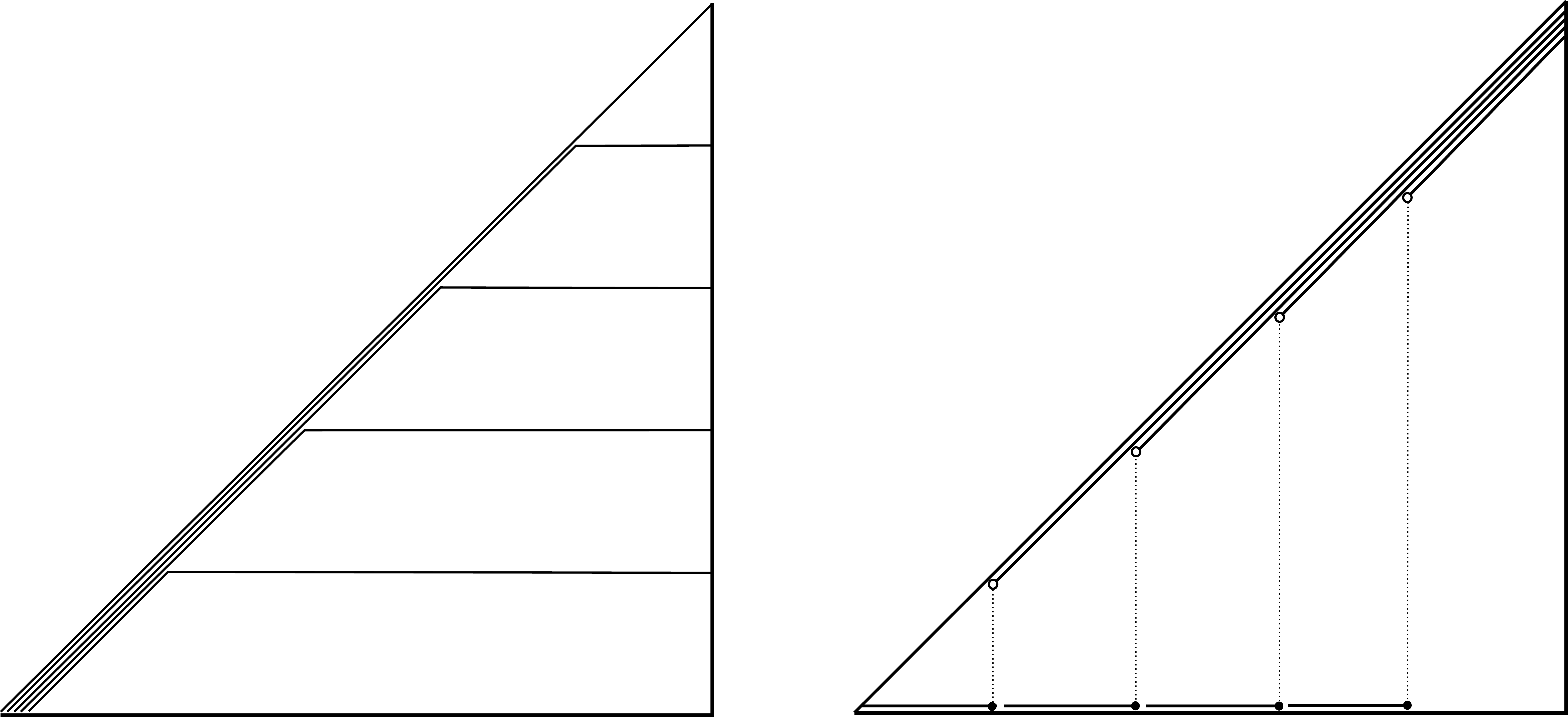}
\caption{The G\" odel and reversed G\" odel composition tomonoid, where the base set is a real interval.}
\label{fig:Goedel-Dreiecksalgebren}
\end{center}
\end{figure}

We note next that, as part of condition (G), the homomorphism from the filter to the composition tomonoids are in the present case provided by construction. Thus we may proceed by specifying the translations by the elements not belonging to the filter.

We will see that, for two $F$-classes $R$ and $S$, $\,\Lambda^{R,S}$ is again largely determined by $\Lambda^R$ and $\Lambda^S$. By Theorem \ref{thm:Goedel-Dreiecke}, there are essentially only two possibilities how each of the latter composition tomonoids may look like. In addition, we may make a distinction according to the presence of border elements. Indeed, the base set of a G\" odel composition tomonoid has a smallest element if and only if $F$ has; similarly, the base set of a reversed G\" odel composition tomonoid has a largest element if and only if $F$ has.

Let us compile the impossible and the trivial cases. The only impossible case is seen from the following proposition.

\begin{proposition} \label{prop:impossible-Goedel-LambdaRS}
Let $R, T \in {\mathcal P}$, where $T < F$, be $\und$-maximal, let $S = R \und T$, and assume that neither $R$ nor $S$ is a singleton. If $\Lambda^R$ is reversed G\" odel and $\Lambda^S$ is G\" odel, $S$ possesses a smallest element.
\end{proposition}

\begin{proof}
By Lemma \ref{lem:coextension-Quadrate}(i), if $R$ possesses a smallest element, so does $S$.
\end{proof}

Again, we call $\Lambda^{R,S}$ trivial if this set contains a single constant mapping, its value being the smallest element of $S$.

\begin{proposition} \label{prop:trivial-Goedel-LambdaRS}
Let $R, T \in {\mathcal P}$ be such that $T < F$ and let $S = R \und T$. Then $\Lambda^{R,S}$ is trivial in each of the following cases:
\AufzAnfang
\Nummer{i} The pair $R, T$ is not $\und$-maximal.

\Nummer{ii} $R$ is a singleton.

\Nummer{iii} $S$ is a singleton.

\Nummer{iv} $\Lambda^R$ is reversed G\" odel and $\Lambda^S$ is G\" odel.
\AufzEnde
\end{proposition}

\begin{proof}
(i) holds by Lemma \ref{lem:coextension-Quadrate}(ii).

(ii) holds by Lemma \ref{lem:coextension-Quadrate}(i),(ii).

(iii) is obvious.

(iv) By part (i)--(iii), it suffices to show the claim in the case that $R, T$ is a $\und$-maximal pair and neither $R$ nor $S$ is a singleton. Then, by Proposition \ref{prop:impossible-Goedel-LambdaRS}, $S$ possesses a smallest element $p$.

Let $r \in R$ and assume that, for some $t \in T$, $\lambda^{R,S}_t(r) = \lambda_t(r) > p$. Let $f \in F$ be small enough, but not the smallest element of $F$, such that $\nu_S(f) < \lambda_t(r)$. Then $\nu_R(f)$ is not the largest element of $R$; hence there is an $r' \in R$ be such that $r' \geq r$ and $r' > \nu_R(f)$. Then $\lambda_f(\lambda_t(r')) = \nu_S(f) < \lambda_t(r')$ because $\lambda_t(r') \geq \lambda_t(r) > \nu_S(f)$, but $\lambda_t(\lambda_f(r')) = \lambda_t(r')$ because $r' > \nu_R(f)$ and hence $\lambda_f(r') = r'$. This contradicts the commutativity of $\lambda_f$ and $\lambda_t$ and we conclude that $\lambda^{R,S}_t = \constant{R}{p}$.
\end{proof}

We determine now the sets of mappings $\Lambda^{R,S}$ in the non-trivial cases. As in the previous section, we will use unit coordinates. In view of (G), this means that we will identify $R$ with one of $(0,1]$ or $[0,1]$ if $\Lambda^R$ is G\" odel, and with one of $[0,1)$ or $[0,1]$ if $\Lambda^R$ is reversed G\" odel; similarly for $S$. We furthermore assume that the mapping $\nu_R$ is $(0,1] \to (0,1], \, x \mapsto x$ or $(0,1] \to [0,1), \, x \mapsto 1-x$, respectively, if $F$ has no smallest element, and similarly for the case that $F$ has a smallest element as well as for $\nu_S$.

\begin{theorem} \label{thm:Goedel-LambdaRS}
Let $R, T \in {\mathcal P}$, where $T < F$, be $\und$-maximal, let $S = R \und T$, and assume that neither $R$ nor $S$ is a singleton.
\AufzAnfang
\Nummer{i} Let $\Lambda^R$ and $\Lambda^S$ be G\" odel. Then, under the identification of $R$ and $S$ with $(0,1]$ or $[0,1]$, $\;\Lambda^{R,S}$ consists of the mappings
\[ \lambda_z \colon R \to S
             \komma r \mapsto r \wedge z, \]
where $z \in \{ s \in S \colon s \leq m \}$ for some $m \in S$.
\Nummer{ii} Let $\Lambda^R$ be G\" odel and let $\Lambda^S$ be reversed G\" odel. Then, under the identification of $R$ and $S$ with $[0,1]$ or $(0,1]$ and $[0,1]$ or $[0,1)$, respectively, $\Lambda^{R,S}$ consists of the mappings
\[ \lambda_z \colon R \to S
             \komma r \mapsto \begin{cases} 0 & \text{if $r \leq 1-z$,} \\
                                            z & \text{if $r > 1-z$,}
                              \end{cases} \]
where $z \in S'$ for some $S' \subseteq S$ containing $0$.
\Nummer{iii} Let $\Lambda^R$ and $\Lambda^S$ be reversed G\" odel. Then, under the identification of $R$ and $S$ with $[0,1)$ or $[0,1]$, $\;\Lambda^{R,S}$ consists of the mappings
\[ \lambda_z \colon R \to S
             \komma r \mapsto \begin{cases} 0 & \text{if $r \leq z$,} \\
                                            r & \text{if $r > z$,}
                              \end{cases} \]
where either $z \in S$ or $z \in \{ s \in S \colon s \leq m \}$ for some $m \in S$.
\AufzEnde
\end{theorem}

\begin{proof}
We assume that $F$ does not possess a smallest element; the procedure is otherwise similar. Using unit coordinates, we identify $F$ with $(0,1]$.

(i) $\Lambda^R$ and $\Lambda^S$ are both $\{ \gamma^{(0,1]}_f \colon f \in (0,1] \}$, where $\gamma^{(0,1]}_f \colon (0,1] \to (0,1] \komma x \mapsto x \wedge f$. According to (\ref{fml:commuting-action}), we have to determine the mappings $\lambda \colon (0,1] \to (0,1]$ such that
\[ \lambda \circ \gamma^{(0,1]}_f \;=\; 
   \gamma^{(0,1]}_f \circ \lambda, \quad f \in (0,1]. \]
That is, $\lambda(r \wedge f) = \lambda(r) \wedge f$ for $r, f \in (0,1]$. Setting $z = \lambda(1)$, we conclude that the solutions are $\lambda(r) = z \wedge r$ for any $z \in (0,1]$.

It follows that each element of $\Lambda^{R,S}$ is of the form $\lambda_z$ for some $z \in (0,1]$. Moreover, by Lemma \ref{lem:coextension-Quadrate}(i)(b), there is a largest $m \in (0,1]$ such that $\lambda_m \in \Lambda^{R,S}$. As $\gamma^{(0,1]}_f \circ \lambda_m \in \Lambda^{R,S}$ for any $f \in (0,1]$ as well, we have that $\lambda_z \in \Lambda^{R,S}$ for all $z \leq m$. The proof of part (i) is complete.

(ii) $\Lambda^R$ is $\{ \gamma^{(0,1]}_f \colon f \in (0,1] \}$ and $\Lambda^S$ is $\{ \beta^{[0,1)}_{1-f} \colon f \in (0,1] \}$. Here, for $t \in [0,1)$, we define
\[ \beta^{[0,1)}_t \colon [0,1) \to [0,1) \komma
x \mapsto \begin{cases} 0 & \text{if $x \leq t$,} \\
x & \text{otherwise.} \end{cases} \]
We have to find $\lambda \colon (0,1] \to [0,1)$ such that
\[ \lambda \circ \gamma^{(0,1]}_f \;=\; 
   \beta^{[0,1)}_{1-f} \circ \lambda, \quad f \in (0,1]. \]
This means $\beta^{[0,1)}_{1-f}(\lambda(r)) = \lambda(f \wedge r)$ for any $r, f \in (0,1]$. Setting $z = \lambda(1)$, we conclude that the solutions are
\[ \lambda(r) = \beta^{[0,1)}_{1-r}(z) =
\begin{cases} 0 & \text{if $r \leq 1-z$,} \\
z & \text{otherwise} \end{cases} \]
for any $z \in [0,1)$.

Hence $\Lambda^{R,S}$ consists of mappings of the form $\lambda_z$, $z \in [0,1)$. For any $z \in (0,1)$, we have $\lambda_z \circ \gamma^{(0,1]}_f = \constant{(0,1]}{0}$ if $f \leq 1-z$, hence $\lambda_0 \in \Lambda^{R,S}$.

(iii) $\Lambda^R$ and $\Lambda^S$ are both $\{ \beta^{[0,1)}_{1-f} \colon f \in (0,1] \}$. We have to find $\lambda \colon [0,1) \to [0,1)$ such that
\[ \lambda \circ \beta^{[0,1)}_{1-f} \;=\; \beta^{[0,1)}_{1-f} \circ \lambda,
   \quad f \in (0,1].\]
We conclude that, for each $r \in (0,1)$, either $\lambda(r) = 0$ or $\lambda(r) = r$. Furthermore, each $\lambda \in \Lambda^{R,S}$ is sup-preserving. We conclude that $\lambda = \beta^{[0,1)}_z$ for some $z \in [0,1)$. The rest follows similarly as in part (i).
\end{proof}

An illustration of Theorem \ref{thm:Goedel-LambdaRS} can be found in Figure \ref{fig:Goedel-Quadratalgebren}.

\begin{figure}[ht!]
\begin{center}
\includegraphics[width=0.45\textwidth]{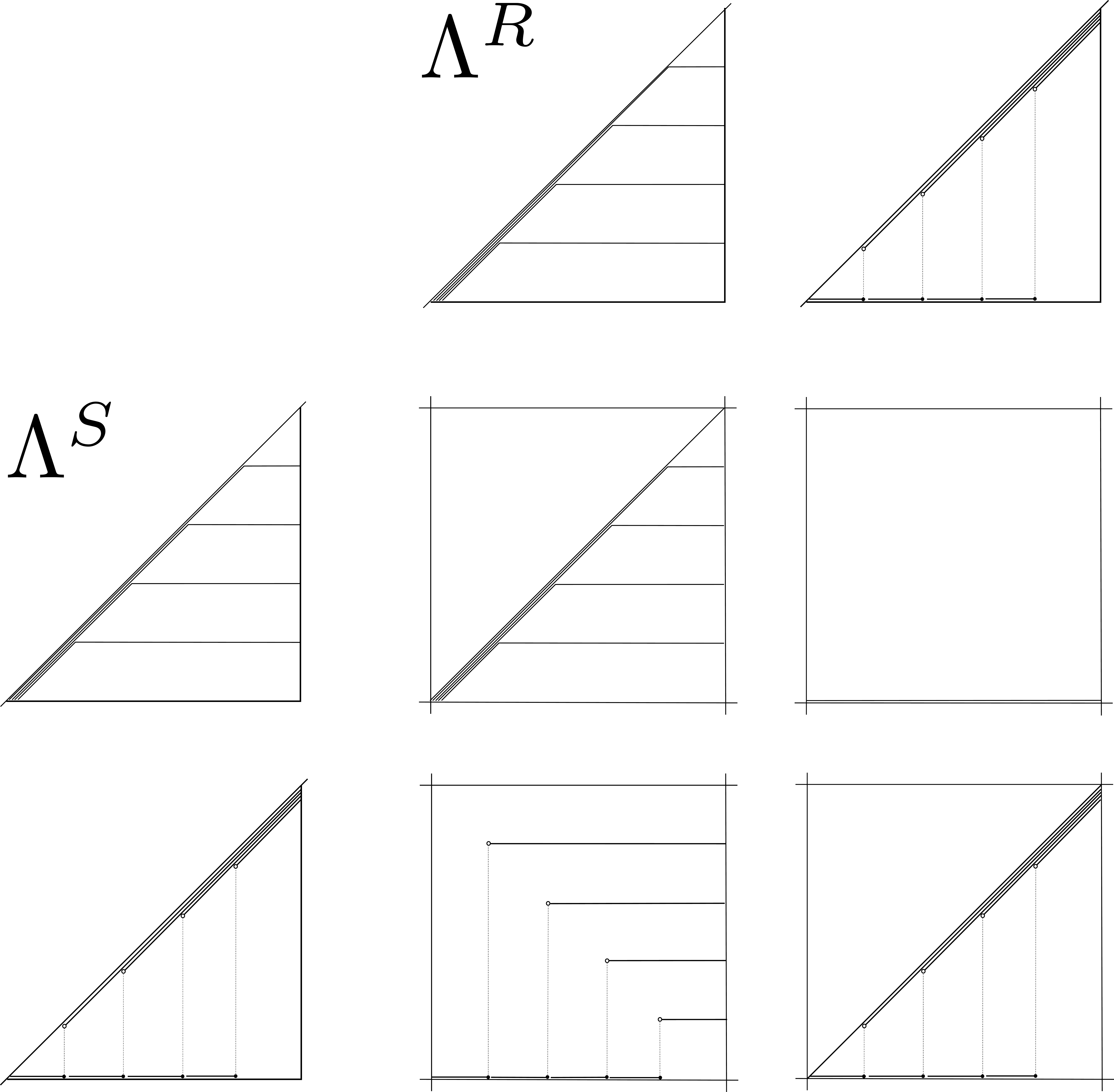}
\caption{A qualitative view on the sets $\Lambda^{R,S}$ depending on $\Lambda^R$ and $\Lambda^S$, in the context of a semilattice real coextension. $\Lambda^{R,S}$ is a subset of the indicated set of mappings, respectively.}
\label{fig:Goedel-Quadratalgebren}
\end{center}
\end{figure}

\section{Examples}
\label{sec:examples}

We are finally in the position to present a number of examples, showing how the theory described in the previous two sections complies with our actual aims, the construction of left-continuous t-norms. To this end, we investigate Archimedean and semilattice coextensions of specific q.n.c.\ tomonoids, choosing the congruence classes such that the new universe is order-isomorphic to the closed real unit interval.

The construction of t-norms has been quite an active research field and has often been motivated by geometric considerations. Apart from the simplest way of generating t-norms from given ones, the ordinal sum, we can mention the rotation, rotation-annihilation, and triple rotation construction \cite{Jen1,Jen2,MaBa1,MaBa2} as well as H-transforms \cite{Mes2}. For a discussion of these constructions from an algebraic point of view, see, e.g., \cite{NEG2}. In order to see how the constructions can be understood within the present framework, see \cite{Vet3}.

In the latter paper \cite{Vet3}, a number of examples of Archimedean real coextensions were already provided. Hence we will focus here on those aspects that are in the present paper newly exhibited. We moreover demonstrate some interesting aspects of coextensions by a semilattice.

We start with an easy example. Let $L_3$ be the three-element \L ukasiewicz chain; its Cayley tomonoid is depicted in Figure \ref{fig:t-Norm-1-2} (left top). We extend $L_3$ by the semilattice $(0,1]$; to this end, we expand the bottom element to a chain that is left-closed right-open and on which the action of the filter is reversed G\" odel. We are led to the nilpotent minimum t-norm $\und_1$; cf.\ again Figure \ref{fig:t-Norm-1-2} (left bottom):
\[ a \und_1 b \;=\;
\begin{cases} a \wedge b & \text{if } a + b > 1, \\
                       0 & \text{otherwise,}
\end{cases} \]
where $a, b \in [0,1]$.

\begin{figure}[ht!]
\begin{center}
\includegraphics[width=0.8\textwidth]{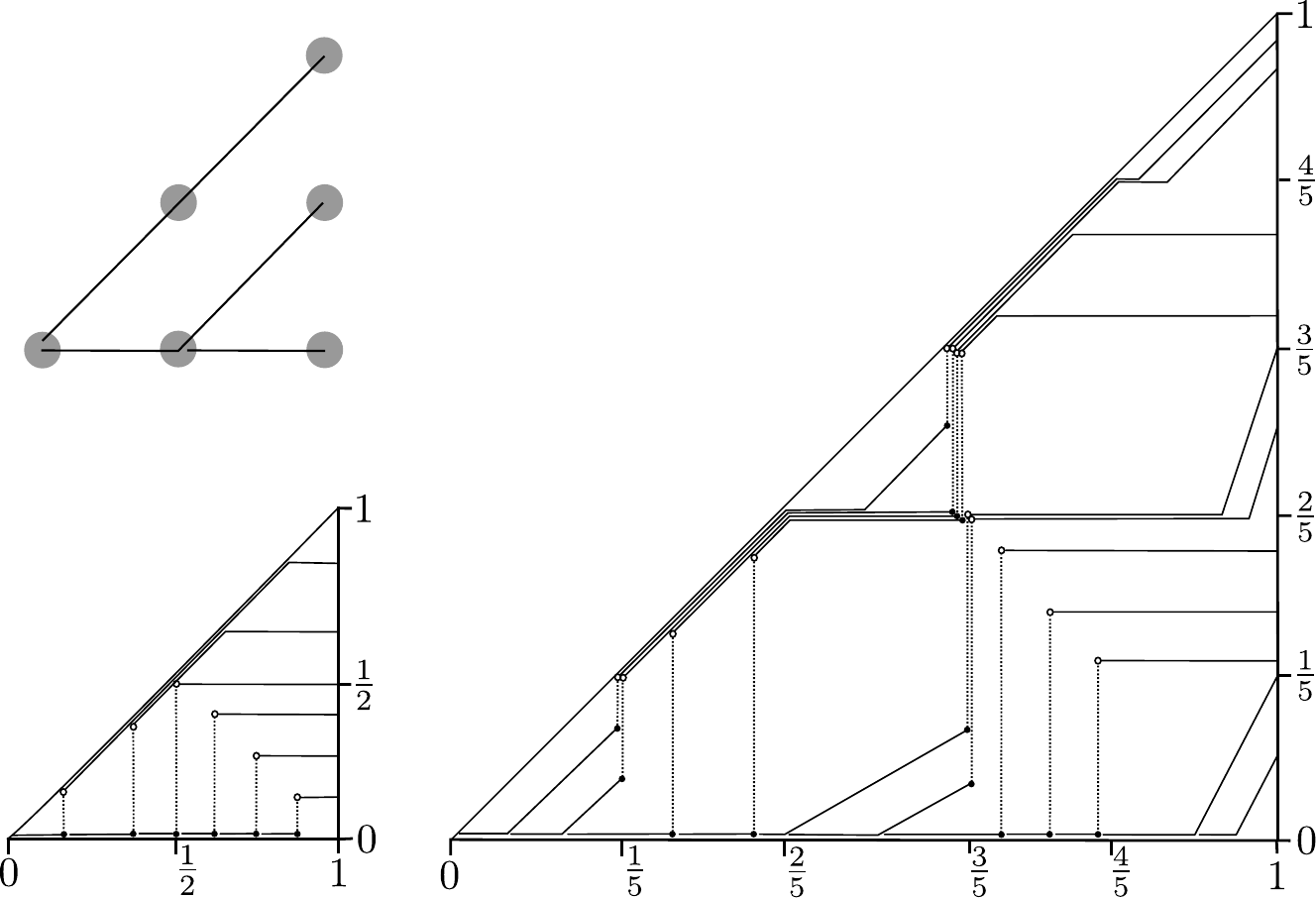}
\caption{The Cayley tomonoids of $L_3$ and the t-norm monoids based on $\und_1$, $\und_2$.}
\label{fig:t-Norm-1-2}
\end{center}
\end{figure}

Next, we extend the t-norm monoid $([0,1]; \leq, \und_1, 1)$ once again. The extending filter is this time the \L ukasiewicz tomonoid and the elements $\frac 1 2$ and $0$ are both expanded to (non-trivial) closed intervals. We choose the coefficients of the homomorphism of the filter to the respective composition tomonoid to be $3$ and $2$, respectively. The resulting t-norm is given as follows; cf.\ Figure \ref{fig:t-Norm-1-2} (right):
\[ a \und_2 b \;=\;
\begin{cases}
(a + b - 1) \vee \frac 4 5
  & \text{\rm if } a, b \geq \frac 4 5, \\
a & \text{\rm if } b \geq \frac 4 5, \text{\rm and } \frac 1 5 < a < \frac 2 5 
    \text{ \rm or } \frac 3 5 < a < \frac 4 5, \\
(a + 3b - 3) \vee \frac 2 5
  & \text{\rm if } b \geq \frac 4 5 
    \text{ \rm and } \frac 2 5 \leq a \leq \frac 3 5, \\
(a + 2b - 2) \vee 0
  & \text{\rm if } b \geq \frac 4 5 \text{ \rm and } a \leq \frac 1 5, \\
a \wedge b
  & \text{\rm if } \frac 3 5 < a, b < \frac 4 5, \\
\frac 2 5
  & \text{\rm if } \frac 3 5 < b < \frac 4 5
    \text{ \rm and } \frac 2 5 \leq a \leq \frac 3 5, \\
a
  & \text{\rm if } \frac 3 5 < b < \frac 4 5
    \text{ \rm and } \frac 1 5 < a < \frac 2 5
    \text{ \rm and } a + b > 1, \\
0
  & \text{\rm if } \frac 3 5 < b < \frac 4 5
    \text{ \rm and } \frac 1 5 < a < \frac 2 5
    \text{ \rm and } a + b \leq 1, \\
0
  & \text{\rm if } \frac 3 5 < b < \frac 4 5
    \text{ \rm and } a \leq \frac 1 5, \\
\frac 2 3 (a + b - 1) \vee 0
  & \text{\rm if } \frac 2 5 \leq a, b \leq \frac 3 5, \\
0
  & \text{\rm if } b \leq \frac 3 5
    \text{ \rm and } a < \frac 2 5
\end{cases} \]
for $a, b \in [0,1]$. As usual, we specify t-norms such that the full definition is achieved only by making use of the commutativity. Note the difference between, say, $\Lambda^{[\frac 2 5, \frac 3 5], [0, \frac 1 5]}$ and the corresponding entry in Figure \ref{fig:Quadratalgebren}.

Our next example demonstrates that, for some congruence classes $R$ and $S$ of a real coextension, $\Lambda^{R,S}$ can be properly contained in the set $\Xi^{R,S}$ as specified in Lemma \ref{lem:Xi}, even if the latter set is not bounded from above. We start with the five-element tomonoid $\mathcal L$ specified in Figure \ref{fig:t-Norm-3} (left). We extend $\mathcal L$ by the product tomonoid and we choose the equivalence classes to be left-closed right-open, a singleton, and three times left-open right-closed, respectively. The following t-norm may arise; cf.\ Figure \ref{fig:t-Norm-3} (right):
\[ a \und_3 b \;=\; 
\begin{cases}
4ab - 3a - 3b + 3
  & \text{\rm if } a, b > \frac 3 4, \\
4ab - 3a - 2b + 2
  & \text{\rm if } b > \frac 3 4 \text{ \rm and } \frac 1 2 < a \leq \frac 3 4, \\ 
4ab - 3a - b + 1
  & \text{\rm if } b > \frac 3 4 \text{ \rm and } \frac 1 4 < a \leq \frac 1 2, \\ 
\frac{a+b-1}{4b-3} \vee 0
  & \text{\rm if } b > \frac 3 4 \text{ \rm and } a \leq \frac 1 4, \\ 
\frac 2 3 (2ab - a - b + \frac 7 8)
  & \text{\rm if } \frac 1 2 < a, b \leq \frac 3 4, \\
\frac 1 4 (1 - \frac 1{4(4a-1)(2b-1)}) \vee 0
  & \text{\rm if } \frac 1 2 < b \leq \frac 3 4
    \text{ \rm and } \frac 1 4 < a \leq \frac 1 2, \\
0 & \text{\rm if } b \leq \frac 3 4
    \text{ \rm and } a \leq \frac 1 4, \\
  & \text{\rm or } a, b \leq \frac 1 2. \\
\end{cases} \]
Note that we have chosen $\Lambda^{[\frac 1 2, \frac 3 4], [\frac 1 4, \frac 1 2]}$, $\Lambda^{[\frac 1 2, \frac 3 4], [0, \frac 1 4]}$, and $\Lambda^{[\frac 1 4, \frac 1 2], [0, \frac 1 4]}$ to contain a proper subset of the set of all mappings that could be included according to Theorem \ref{thm:LambdaRS-product}.

\begin{figure}[ht!]
\begin{center}
\includegraphics[width=0.8\textwidth]{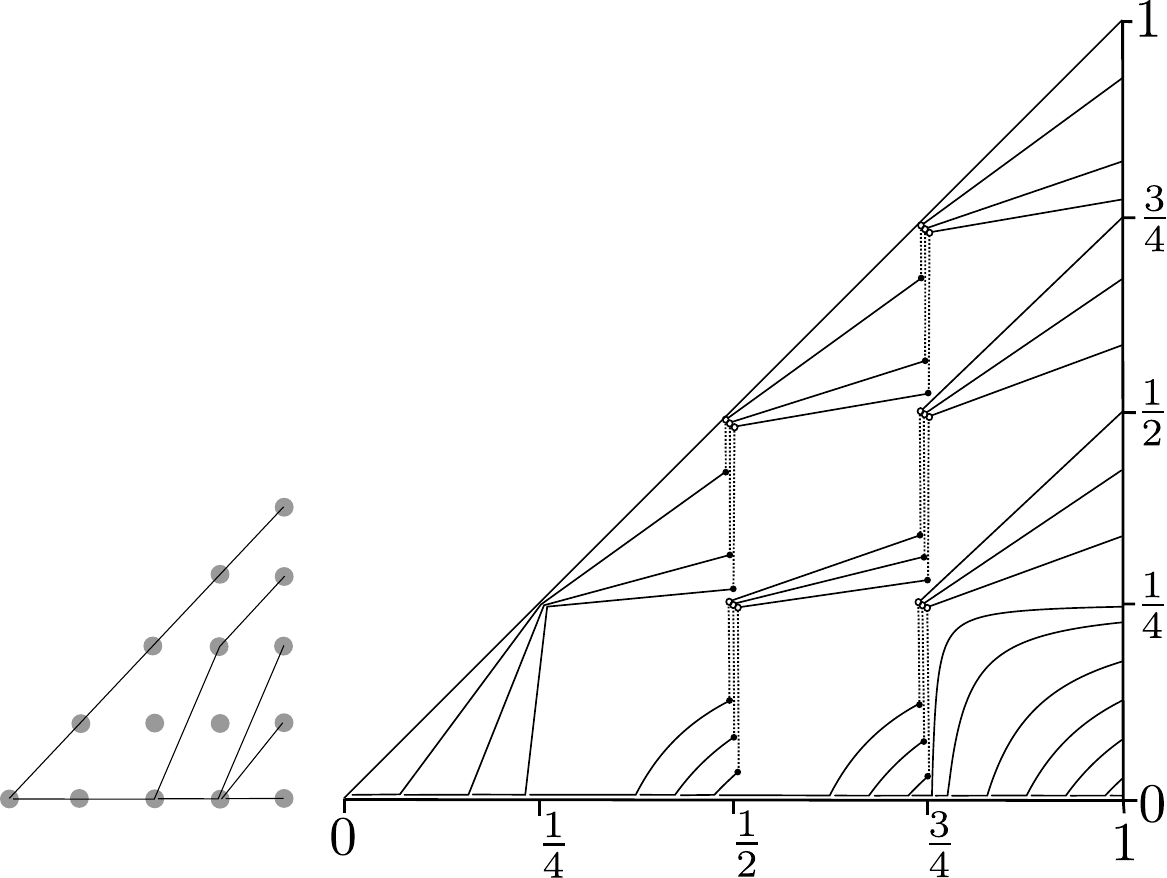}
\caption{The t-norm $\und_3$.}
\label{fig:t-Norm-3}
\end{center}
\end{figure}

We finally indicate a further example of a semilattice coextension, which is less straightforward than $\und_1$ above. We extend the same five-element tomonoid $\mathcal L$ as in the previous example and we use the same congruence classes; but this time the extending filter is the semilattice with the universe $(0,1]$. This may lead to the following t-norm; cf.\ Figure \ref{fig:t-Norm-4}:
\[ a \und_4 b \;=\; 
\begin{cases}
a \wedge b
  & \text{\rm if } a, b > \frac 3 4, \\
a \wedge (b - \frac 1 4)
  & \text{\rm if } b > \frac 3 4 \text{ \rm and } \frac 1 2 < a \leq \frac 3 4, \\ 
a \wedge (b - \frac 1 2)
  & \text{\rm if } b > \frac 3 4 \text{ \rm and } \frac 1 4 < a \leq \frac 1 2, \\ 
a
  & \text{\rm if } b > \frac 3 4 
    \text{ \rm and } a < \frac 1 4
    \text{ \rm and } a + b > 1, \\ 
0
  & \text{\rm if } b > \frac 3 4 
    \text{ \rm and } a < \frac 1 4
    \text{ \rm and } a + b \leq 1, \\ 
(a - \frac 1 4) \wedge (b - \frac 1 4) \wedge \frac 7{16}
  & \text{\rm if } \frac 1 2 < a, b \leq \frac 3 4, \\
\frac 1 8
  & \text{\rm if } \frac 5 8 < b \leq \frac 3 4
    \text{ and } \frac 3 8 < a \leq \frac 1 2, \\
0 & \text{\rm if } b \leq \frac 3 4
    \text{ and } a \leq \frac 3 8, \\
  & \text{\rm or } b \leq \frac 5 8
    \text{ and } a \leq \frac 1 2. \\
\end{cases} \]
Note that, again, $\Lambda^{[\frac 1 2, \frac 3 4], [\frac 1 4, \frac 1 2]}$, $\Lambda^{[\frac 1 2, \frac 3 4], [0, \frac 1 4]}$, and $\Lambda^{[\frac 1 4, \frac 1 2], [0, \frac 1 4]}$ do not contain all the mappings that we could include. In fact $\Lambda^{[\frac 1 2, \frac 3 4], [0, \frac 1 4]}$ and $\Lambda^{[\frac 1 4, \frac 1 2], [0, \frac 1 4]}$ contain only two elements and are hence not even downwards closed subsets of the set of all possible mappings.

\begin{figure}[ht!]
\begin{center}
\includegraphics[width=0.8\textwidth]{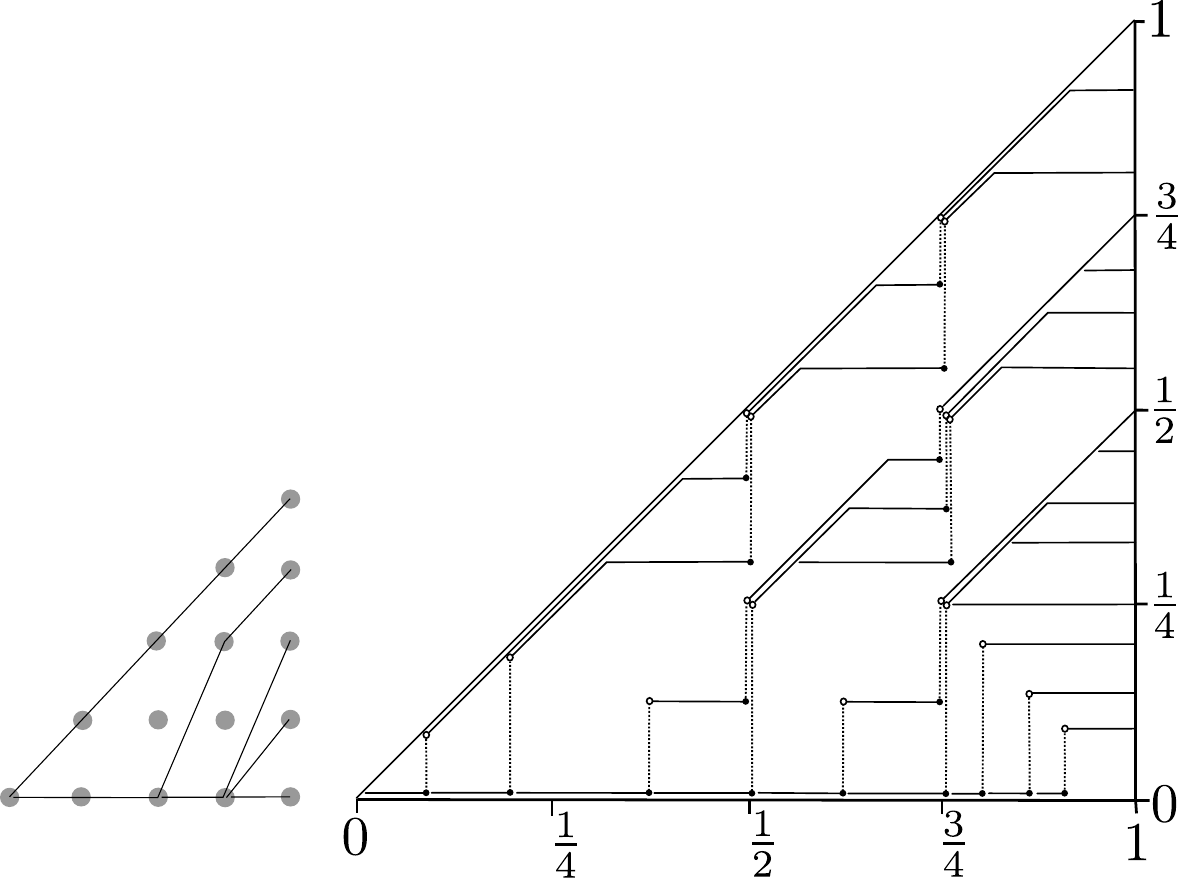}
\caption{The t-norm $\und_4$.}
\label{fig:t-Norm-4}
\end{center}
\end{figure}

\section{Conclusion}
\label{sec:conclusion}

In our previous papers \cite{Vet1,Vet2,Vet3}, we have demonstrated that left-continuous t-norms can be constructed in a way reminding of a jigsaw puzzle, composing triangular or rectangular pieces in order to obtain the Cayley tomonoid associated with a new t-norm. The geometric intuition has an algebraic counterpart: we construct real coextensions of quantic, negative, commutative tomonoids. We have discussed such coextensions, focusing on two situations. On the one hand, we considered the case that the extending filter is Archimedean and on the basis of our results from \cite{Vet2}, we have determined in the present work all relevant pieces needed for the specification of the coextended tomonoid. On the other hand, we have included in the discussion a further type of coextension, namely, the case that the extending filter is a semilattice. Here, a semilattice is a chain endowed with the minimum as the monoidal product.

It should be noted that the theory of Archimedean real coextensions could still be further developed. Although we have provided the means of describing such a coextension, the construction can still be involved. We have seen that the sets $\Lambda^{R,S}$ of restrictions of the translations to a congruence class $R$ in domain and $S$ in range, is to be chosen as a subset of a certain set of mappings from $R$ to $S$; $\Lambda^{R,S}$ can be unbounded but can also be determined by a largest mapping $\zeta \colon R \to S$. The choice of $\zeta$ can in general not be done in an arbitrary way and it remains to determine the exact possibilities.

A continuation of our work is possible in several respects. On the one hand, the class of covered t-norms can be further enlarged, e.g., by including the inverse limit of tomonoids, in the sense demonstrated in \cite{Mes2}. On the other hand, it is open if a larger class of operations is accessible to our method. For instance, uninorms might be examined in the present framework.

\subsubsection*{Acknowledgement}

The author acknowledges the support by the Austrian Science Fund (FWF): project I 1923-N25. He is moreover grateful to the anonymous reviewers for their valuable comments and useful suggestions, which helped to improve this paper.

\end{document}